\documentclass[12pt]{article}

\usepackage{amsmath}
\usepackage{amssymb}
\usepackage{amsthm}
\usepackage{amsfonts}
\usepackage{amscd}
\usepackage{graphicx} 
\usepackage{hyperref}
\usepackage{array} 
\usepackage{enumitem} 
\usepackage[numbers,sort&compress]{natbib}
\usepackage{titling} 
\usepackage[T1]{fontenc} 
\usepackage{tikz}
\usepackage[medium]{titlesec}
\usepackage{longtable}
\usepackage{etoolbox} 
\usepackage{geometry}
\usepackage{setspace}
\usepackage{verbatim} 
\usepackage{adjustbox}
\usepackage{caption}

\hypersetup{colorlinks,urlcolor=black!50!green,citecolor=black!45!green,linkcolor=blue}
\makeatletter\patchcmd{\ttlh@hang}{\parindent\z@}{\parindent\z@\leavevmode}{}{}\patchcmd{\ttlh@hang}{\noindent}{}{}{}\makeatother 
\titlespacing*{\section}{0pt}{2mm}{2mm}
\titlespacing*{\subsection}{0pt}{2mm}{2mm}
\titlespacing*{\paragraph}{0pt}{-2mm}{0mm}
\newcommand{\myspace}{\setlength{\abovedisplayskip}{2mm}\setlength{\belowdisplayskip}{2mm}}

\newif\ifshowtikz
\showtikztrue\newcommand{\fig}[3]{\includegraphics[height=#1cm, width=#2cm]{img/#3}}
\showtikztrue\newcommand{\udot}[1]{\tikz[baseline=(todotted.base)]{\node[inner sep=1pt,outer sep=0pt] (todotted) {$#1$};\draw[densely dotted] (todotted.south west) -- (todotted.south east);}}
\let\oldtikzpicture\tikzpicture
\let\oldendtikzpicture\endtikzpicture
\renewenvironment{tikzpicture}{\ifshowtikz\expandafter\oldtikzpicture\else\comment\fi}{\ifshowtikz\oldendtikzpicture\else\endcomment\fi}

\theoremstyle{plain}
\newtheorem{lemma}{Lemma}
\newtheorem{proposition}{Proposition}
\newtheorem{theorem}{Theorem}
\newtheorem{corollary}{Corollary}

\newtheorem{theoremext}{Theorem}

\theoremstyle{definition}
\newtheorem{definition}{Definition}
\newtheorem{remark}{Remark}
\newtheorem{example}{Example}
\newtheorem{notation}{Notation}
\newtheorem{algorithm}{Algorithm}
\newtheorem{problem}{Problem}

\newenvironment{Mlist}{\begin{itemize}[topsep=0pt,itemsep=0pt,leftmargin=7mm]}{\end{itemize}}
\newenvironment{Menum}{\begin{enumerate}[topsep=0pt,itemsep=0pt,leftmargin=7mm]}{\end{enumerate}}
\newenvironment{MenumA}{\begin{enumerate}[topsep=0pt,itemsep=0pt,leftmargin=7mm,label=\alph*.]}{\end{enumerate}}

\newcommand{\PRP}[1]{Proposition~\ref{prp:#1}}
\newcommand{\LEM}[1]{Lemma~\ref{lem:#1}}
\newcommand{\THM}[1]{Theorem~\ref{thm:#1}}
\newcommand{\COR}[1]{Corollary~\ref{cor:#1}}
\newcommand{\DEF}[1]{Definition~\ref{def:#1}}
\newcommand{\RMK}[1]{Remark~\ref{rmk:#1}}
\newcommand{\SEC}[1]{\textsection\ref{sec:#1}}
\newcommand{\FIG}[1]{Figure~\ref{fig:#1}}
\newcommand{\ALG}[1]{Algorithm~\ref{alg:#1}}
\newcommand{\TAB}[1]{Table~\ref{tab:#1}}
\newcommand{\EXM}[1]{Example~\ref{exm:#1}}
\newcommand{\NTN}[1]{Notation~\ref{ntn:#1}}
\newcommand{\PRO}[1]{Problem~\ref{pro:#1}}
\newcommand{\END}{\hfill $\vartriangleleft$}
\newcommand{\df}[1]{\textit{#1}}
\newcommand{\mclaim}[1]{\item[\textbf{#1)}]}
\newcommand{\refmclaim}[1]{#1)}
\newcommand{\Iff}{if and only if }
\newcommand{\st}{{such that }}
\newcommand{\Wlog}{without loss of generality }
\newcommand{\resp}{respectively}
\newcommand{\wrt}{with respect to }
\newcommand{\arrow}[3]{#1\stackrel{#2}{\longrightarrow}#3}
\renewcommand{\c}{\colon}
\newcommand{\dto}{\dashrightarrow}
\newcommand{\md}[1]{\langle #1\rangle}
\newcommand{\set}[2]{\{ #1 ~|~ #2 \}}
\newcommand{\Sum}[2]{\overset{#2}{\underset{#1}{\sum}}}
\newcommand{\aut}{\operatorname{Aut}}
\newcommand{\rnk}{\operatorname{rank}}
\newcommand{\C}{\mathbb{C}}
\newcommand{\R}{\mathbb{R}}
\newcommand{\Q}{\mathbb{Q}}
\newcommand{\Z}{\mathbb{Z}}
\renewcommand{\P}{\mathbb{P}}
\renewcommand{\S}{\mathbb{S}}

\newcommand{\CG}{\mathcal{G}}
\newcommand{\CH}{\mathcal{H}}
\newcommand{\CL}{\mathcal{L}}

\newcommand{\CT}{\mathcal{T}}
\newcommand{\CV}{\mathcal{V}}
\newcommand{\CW}{\mathcal{W}}
\newcommand{\FD}{\mathfrak{D}}
\newcommand{\FC}{\mathfrak{C}}
\newcommand{\BF}{\mathbf{F}}
\newcommand{\BG}{\mathbf{G}}
\newcommand{\BP}{\mathbf{P}}
\newcommand{\BS}{\mathbf{S}}
\newcommand{\BN}{\mathbf{N}}
\newcommand{\e}{e}
\newcommand{\p}{\varepsilon}
\renewcommand{\k}{k}
\renewcommand{\l}{\ell}
\newcommand{\h}{h}

\title{Webs of rational curves on real surfaces and a classification of real weak del Pezzo surfaces}
\author{Niels Lubbes}
\date{\today}

\begin{document}
\myspace

\maketitle

\begin{abstract}
We classify webs of minimal degree rational 
curves on surfaces and give a criterion for 
webs being hexagonal. 
In addition, we classify 
Neron-Severi lattices of real weak del Pezzo surfaces.
These two classifications are related
to root subsystems of E8.
\\[2mm]
\\[2mm]
\end{abstract}

\begingroup
\def\addvspace#1{\vspace{0pt}}
\tableofcontents
\endgroup

\section{Introduction}
\label{sec:intro}

Lines play a central role in Euclidean geometry and are rational curves of minimal degree. 
Through any 
two points in the plane exists a unique line and the family of lines in the plane is 2-dimensional. 
We investigate the geometry of real surfaces in projective space,
by considering surfaces as a union of curves that are ``simple''.
With \df{surface} we shall mean a real irreducible algebraic surface (see \SEC{pre}).
A \df{simple family} 
is an algebraic family of minimal degree rational curves
that covers a surface $X\subset\P^n$,
\st a general curve in this family is smooth outside the singular locus of $X$.
Moreover, we assume that the dimension of a simple family is as large as possible
(see \DEF{fam}). 
A \df{simple curve} is a curve that belongs to some simple family.

The \df{intersection product} of two simple families that cover $X$
is defined as the number of intersections between a general curve in the first family 
and a general curve in the second family, outside the singular locus of $X$.

The \df{simple family graph} $\CG(X)$ is defined as follows:
\begin{Mlist}
\item Each vertex is a simple family of $X$. 
A vertex is labeled with the dimension of the simple family.

\item
We draw between two simple families
an edge if their intersection product is at least two. 
Such an edge is labeled with this intersection product.
\end{Mlist}
Notice that simple families in $\CG(X)$ that are not
connected by an edge must have intersection product one.

See \FIG{simple} for examples of simple family graphs.
The simple family graph of a one-sheeted hyperboloid
was discovered in 1669 by Sir Christopher Wren~\cite{wrn1}.
The simple family graph of a ring cyclide was discovered by Yvon Villarceau
and the two neighboring vertices define families of \df{Villarceau circles}~\citep[1848]{vil1}.
The simple family graph of a Blum cyclide
was discovered in 1980 by Richard Blum~\cite{blum1}.
The ring cyclide, Blum cyclide and dP6 are 
examples of weak del Pezzo surfaces of degree 4, 4 and 6, respectively (see \DEF{wdp}).

In this article we address the following problem.
\begin{problem}
\label{pro:1}
{\it
Classify simple family graphs $\CG(X)$ of real surfaces
$X\subset\P^n$ and
determine properties of $X$ 
that are encoded in the invariant $\CG(X)$.
}
\end{problem}

In the following \THM{coarse}, $[\Lambda]^2$ denotes the set of $2$-element subsets of 
$\Lambda:=\{1,\ldots,m\}$ for some $m\geq 3$.
If $G$ and $H$ are non-empty graphs, then $G\diamond H$
is the graph that is defined as the union $G\cup H$, where 
in addition each vertex in~$G$ is connected with each vertex in~$H$
via an edge with label~$2$.

\begin{theorem}
\label{thm:coarse}
If $\CG(X)$ is the simple family graph of a real surface~$X\subset\P^n$,
then either $\CG(X)=\emptyset$ or $\CG(X)$ is characterized by one of the following graphs:
\begin{MenumA}

\item 
The minimal family graph of a real weak del Pezzo surface
which is encoded by a row in \TAB{nsl} and where
each vertex has label~1. 

\item 
The graph that consists of a single vertex with label 3.

\item 
The graph that consists of a single vertex with label 2.

\item 
A graph that consists of vertices labeled 1 and no edges.

\item 
A completely connected graph \st each vertex has label 1 and each edge has label 2. 

\item 
The graph with vertex set $[\Lambda]^2$ \st each vertex has label 1
and an edge between different vertices $A, B\in[\Lambda]^2$ has label $4-2|A\cap B|$.

\item 
The graph with vertex set $[\Lambda]^2$ 
\st there is an edge between different vertices $A, B\in[\Lambda]^2$ if and only if $|A\cap B|=0$.
Each vertex has label 1 and each edge has label 2.

\item
$G\diamond H$ where $G$ and $H$ are characterized by (d) and (e), \resp.

\item
$G\diamond H$ where $G$ and $H$ are characterized by (d) and (f), \resp.

\item
$G\diamond H$ where $G$ and $H$ are characterized by (e) and (g), \resp.

\item
$G\diamond H$ where $G$ and $H$ are characterized by (c) and (e), \resp.

\item
$G\diamond H$ where $G$ and $H$ are characterized by (c) and (f), \resp;
optionally, some vertices with label 1 are omitted.

\end{MenumA}
\end{theorem}

\begin{figure}
\centering
\def\IMG{(0,1)}

\def\CYU{-1.8}
\def\CYL{-2.3}

\definecolor{wdred}{RGB}{153, 25, 25}
\definecolor{wlred}{RGB}{215, 91, 78}
\definecolor{wdblue}{RGB}{14 , 97,120}
\definecolor{wlblue}{RGB}{66 ,143,171}
\definecolor{wgreen}{RGB}{101,137, 23}
\definecolor{wblack}{RGB}{  0,  0,  0}
\definecolor{wbrown}{RGB}{ 86, 34,  0}
\newcommand{\vv}[2][] {\draw[draw=black, fill=#1, line width=0.2mm] #2 circle [radius=2mm] node[yellow] {$1$};}
\newcommand{\vw}[1] {\draw[draw=black, fill=gray!20, line width=0.2mm] #1 circle [radius=2mm] node[gray] {$1$};}
\newcommand{\ee}[3] {\draw[line width=0.2mm, draw=black] #1 -- #2;\node[black] at #3 {$2$};}

\def\sx{0} \def\sy{-0.8}
\def\ddx{0.9} \def\ddy{0.9}
\def\ax{\sx-\ddx} \def\ay{\sy}
\def\bx{\sx     } \def\by{\sy}
\def\cx{\sx+\ddx} \def\cy{\sy}
\def\dx{\sx-\ddx} \def\dy{\sy-\ddy}
\def\ex{\sx     } \def\ey{\sy-\ddy}
\def\fx{\sx+\ddx} \def\fy{\sy-\ddy}
\def\A{(\ax,\ay)}
\def\B{(\bx,\by)}
\def\C{(\cx,\cy)}
\def\D{(\dx,\dy)}
\def\E{(\ex,\ey)}
\def\F{(\fx,\fy)}
\def\AD{(\ax+0.13,\ay-\ddy/2)} 
\def\BE{(\bx+0.13,\by-\ddy/2)} 
\def\CF{(\cx+0.13,\cy-\ddy/2)} 

\def\dddx{0.5}
\def\gx{\sx-\dddx} \def\gy{\sy}
\def\hx{\sx-\dddx} \def\hy{\sy-\ddy}
\def\ix{\sx+\dddx} \def\iy{\sy-\ddy}
\def\jx{\sx+\dddx} \def\jy{\sy}
\def\G{(\gx,\gy)}
\def\H{(\hx,\hy)}
\def\I{(\ix,\iy)}
\def\J{(\jx,\jy)}
\def\GH{(\gx+0.13,\gy-\ddy/2)} 

\begin{tabular}{@{}c@{\quad}c@{\quad}c@{}}
\begin{tikzpicture}
\node[inner sep=0pt] at \IMG {\fig{3}{3}{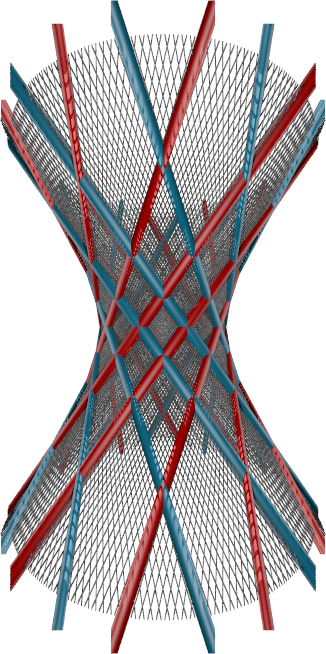}}; 
\vv[wdred]{(\sx-\ddx/2,\sy-\ddy/2)};
\vv[wdblue]{(\sx+\ddx/2,\sy-\ddy/2)};
\node[black] at (0,\CYU) {one-sheeted};
\node[black] at (0,\CYL)  {hyperboloid};
\end{tikzpicture}
&
\begin{tikzpicture}
\node[inner sep=0pt] at \IMG {\fig{3}{4}{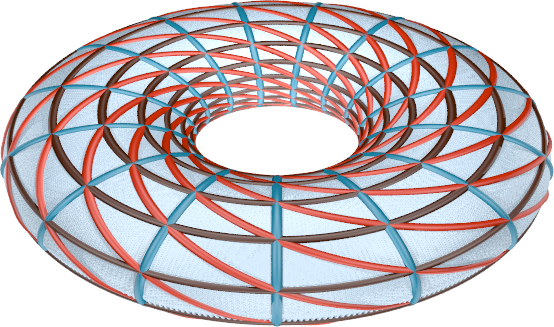}}; 
\ee \G \H \GH;
\vv[wlred]{\G};
\vw{\H};
\vv[wbrown]{\I};
\vv[wlblue]{\J};
\node[black] at (0,\CYL) {ring cyclide};
\end{tikzpicture}
&
\begin{tikzpicture}
\node[inner sep=0pt] at \IMG {\fig{3}{4}{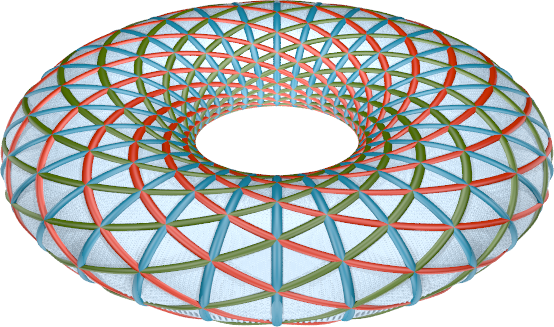}}; 
\ee \G \H \GH;
\vv[wlred]{\G};
\vv[wgreen]{\H};
\vw{\I};
\vv[wlblue]{\J};
\node[white] at (0,\CYL) {};
\end{tikzpicture}
\\[5mm]
\begin{tikzpicture}
\node[inner sep=0pt] at \IMG {\fig{3}{4.4}{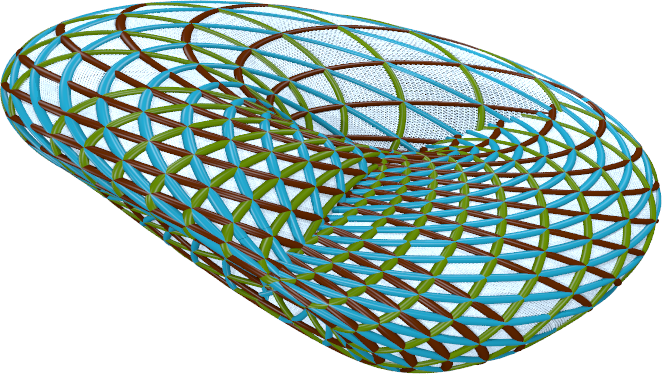}}; 
\vv[wgreen]{(\sx,\sy)};
\vv[wbrown]{(\sx-\ddx/2-0.2,\sy-\ddy)};
\vv[wlblue]{(\sx+\ddx/2+0.2,\sy-\ddy)};
\node[black] at (0,\CYL) {dP6};
\end{tikzpicture}
&
\begin{tikzpicture}
\node[inner sep=0pt] at \IMG {\fig{3}{4}{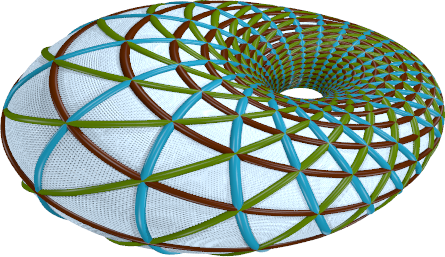}}; 
\ee \A \D \AD;
\ee \B \E \BE;
\ee \C \F \CF;
\vv[wlblue]{\A};
\vv[wgreen]{\B};
\vv[wbrown]{\C};
\vw{\D};
\vw{\E};
\vw{\F};
\node[black] at (0,\CYL) {Blum cyclide};
\end{tikzpicture}
&
\begin{tikzpicture}
\node[inner sep=0pt] at \IMG {\fig{3}{4}{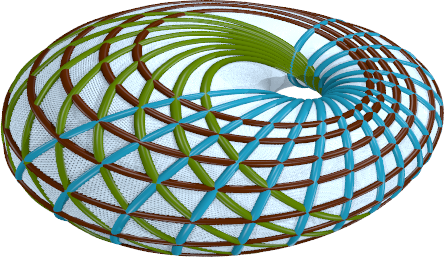}}; 
\ee \A \D \AD;
\ee \B \E \BE;
\ee \C \F \CF;
\vw{\A};
\vw{\B};
\vw{\C};
\vv[wlblue]{\D};
\vv[wgreen]{\E};
\vv[wbrown]{\F};
\node[white] at (0,\CYL) {Blum cyclide};
\end{tikzpicture}
\end{tabular}
\caption{Examples of simple family graphs.}
\label{fig:simple}
\end{figure}

\begin{corollary}
\label{cor:labels}
Suppose that $X\subset\P^n$ is a real surface.
A vertex in $\CG(X)$ has label $\leq 3$ and an edge in $\CG(X)$ has label $\leq 8$.
If an edge has label~$3$ or label~$\geq 5$, then $\CG(X)$ contains $\leq 2160$
vertices and $\leq 2262600$ edges.
\end{corollary}

\newpage
\begin{corollary}
\label{cor:dis}
If $\CG(X)$ has $v\geq 4$ vertices and is characterized by \THM{coarse}a,
then $\CG(X)$ is not characterized by \THM{coarse}[b-l], 
except for the simple family graphs associated to the following rows in \TAB{nsl}:
\begin{Mlist}
\item $v=4$: rows $\{16, 33\}$, $\{116\}$ and $\{283, 284, 288, 289, 647, 671\}$ are also characterized by \THM{coarse}d, \THM{coarse}h and \THM{coarse}e, \resp.
\item $v=5$: row 15 is also characterized by \THM{coarse}d.
\item $v=6$: rows 335, 360, 708, 714, 721, 737, 745 and 761 are also characterized by \THM{coarse}f.
\item $v=7$: rows 267, 268 and 272 are also characterized by \THM{coarse}i.
\end{Mlist}
\end{corollary}

The proof of \THM{coarse} uses results from \cite{nls-algo-min-fam} (see \LEM{bas} and \LEM{pp}).
We recover \citep[Corollary~1a]{nls-algo-min-fam} which addresses the second part of \PRO{1}:

\begin{corollary}
\label{cor:sphere}
Suppose that $X\subset\P^n$ is a real surface.
If a vertex in $\CG(X)$ has label $\geq 3$, then the normalization of 
$X$ is biregular isomorphic to a quadric whose real points form the 
unit-sphere $S^2\subset \R^3$. 
\end{corollary}

\newpage
The unique recovery of $\CG(X)$ from a row in \TAB{nsl}
will be discussed at \THM{nsl} in a moment and later at \EXM{howto}.
Our methods are constructive and we hope to interest both geometric modelers and algebraic geometers. 
With \citep[Algorithm~1]{nls-algo-min-fam} we can compute the graph $\CG(X)$
from a given birational map $\P^2\dto X$. 
The implementation at \citep[\texttt{ns\_lattice}]{ns_lattice} can be used for checking the 850 cases in \TAB{nsl} automatically. 
We find for example that a simple family graph does not consist of 4 vertices and 2 edges.
See \citep[{\tt orbital}]{orbital} for the visualization of simple families on surfaces.

The Cayley-Salmon theorem states that a smooth cubic surface 
over an algebraically closed field contains 27 straight lines \citep[1848]{cay1}
and the intersection graph of the lines is the dual of the Schl\"afli graph.
The set of hyperplane sections of the cubic surface containing at least one of these lines define 27 families of conics.
In the cubic Clebsch surface these 27 families are real \citep[1871]{cleb1} 
and the corresponding simple family graph is in \FIG{cleb}. 
\begin{figure}[!ht]
\centering
\newcommand{\vw}[1] {\draw[draw=black, fill=black!3!green, line width=0.01mm] #1 circle [radius=1.8mm] node[gray] {\tiny $1$};}
\begin{tikzpicture}
\node[inner sep=0pt] at (0,0) {\fig{5.9}{5.9}{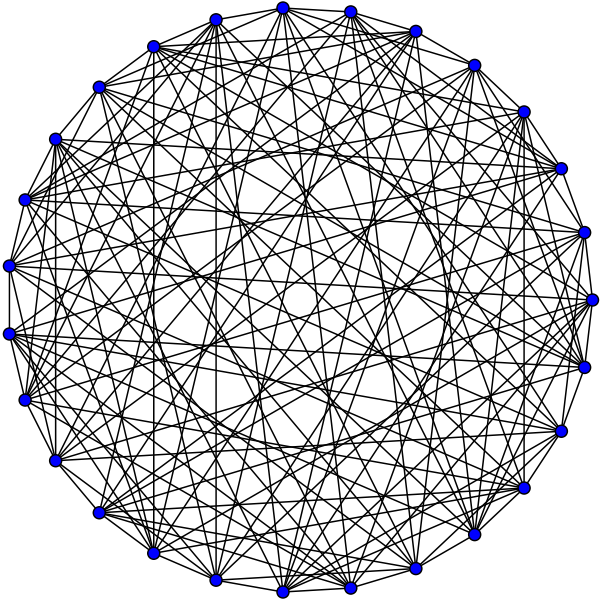}};   
\foreach \a in {1,2,...,27}{\vw{(\a*360/27: 2.9cm)};}
\end{tikzpicture} 
\caption{
The 10-regular graph with 27 vertices is
the simple family graph of a Clebsch surface (the edges are labeled 2).
This graph is known as the generalized quadrangle $GQ(2,4)$ \cite{wiki}.
}
\label{fig:cleb}
\end{figure}

Normal cubic surfaces are special cases of weak del Pezzo surfaces \cite{dol1,man2}.
Intersection graphs of lines on (complex) del Pezzo surfaces have been studied \cite{mo} and
these intersection graphs uniquely determine the simple family graphs of the surfaces.
An advantage of simple family graphs of weak del Pezzo surfaces is that,
unlike intersection graphs of lines,
such graphs apply to all rational surfaces
whose simple family graphs are not characterized by one of
the eleven ``exceptional graph types'' as listed in \THM{coarse}[b-l].

Del Pezzo surfaces are investigated since at least 1830 \citep[Chapter 9, Historical~notes]{dol1}
and the geometry of these celebrated surfaces remain more than 175 years later still a topic of research \cite{top1,sturm}.
Del Pezzo surfaces play a central role in the classification of real rational surfaces 
\citep[Theorem~1.9]{kol2}\citep[Corollary~2.6]{sil1}.
See \cite{sch-dp} for an introduction to the theory of del Pezzo surfaces with emphasis on geometric modeling.

The Neron-Severi lattice~$\BN(X)$ of a weak del Pezzo surface $X\subset \P^n$ encodes 
to a large extend the geometry of $X$ and in particular
its minimal family graph~$\CG(X)$ (see \PRP{class} and \RMK{G}).
It is remarkable that this algebraic structure is determined by the combinatorics of the E8 root system.
See \DEF{nsl} for the data associated to the Neron-Severi lattice
and see \DEF{wdp} for the notion of weak del Pezzo surface.
In light of \THM{coarse}a we classify the Neron-Severi lattices of
real weak del Pezzo surfaces. 

\begin{theorem}
\label{thm:nsl}
If $X$ is a real weak del Pezzo surface, 
then its Neron-Severi lattice $\BN(X)$ 
is up to isomorphism uniquely
characterized by exactly one of 825 rows in \TAB{nsl}.
\end{theorem}

See \EXM{howto} for an explanation of how to read \TAB{nsl} and how to recover $\CG(X)$ uniquely from a row.
\THM{nsl} has been known before with the following additional assumptions:
\begin{Mlist}
\item $X$ is a real (non-weak) del Pezzo surface \citep[Corollary~2.1]{wal1}.
\item the real structure of $X$ acts trivially on $N(X)$ \cite{val1} (see also \citep[Sections~8.7.1 and 8.8.1]{dol1}).
\item $X$ is a real weak del Pezzo surface of degree 3 with only real singularities \cite{mil} (there are 44 cases).
\end{Mlist}

If $X\subset\P^3$ is quadric surface, 
then $\BN(X)$ is isomorphic to the Neron-Severi lattice of either a sphere, a quadric cone or 
a hyperboloid of one-sheet.
In comparison, if $X\subset\P^3$ is a cubic surface with at most isolated singularities, 
then it follows from \THM{nsl} that there are up to isomorphism 56 different choices for
$\BN(X)$ and the number of real lines in $X$ is either $\leq 12$ or in $\{15,16,21,27\}$.

A \df{$\lambda$-web of curves} on a surface $X$ 
is defined as a set $\CW$ of algebraic curves contained in $X$
such that through almost every point in $X$ pass exactly
$\lambda$ curves in $\CW$.
Something remarkable can be observed about the 
3-webs of simple curves as visualized in \FIG{simple}.
In 1927, Thomson and Blaschke  
talked about such webs during their spring walks on Posillipo hill in Italy \citep[Vorwort]{bla1}.
The observation is that an open space on the surface is bordered 
by exactly three simple curves in a discretized realization of the web. In particular, many hexagonal
patterns emerge and therefore such triangular 3-webs are also known
as \df{hexagonal webs} (see \DEF{hex} or \citep[Lecture~18]{omni}). 
Hexagonal webs of simple curves lead to nice triangularizations of the 
underlying surface.
The property for a web to be hexagonal is purely topological, 
and its existence reveals in some cases properties of the surface under study.
See 
\citep[Appendix]{per1} 
for an overview of historical and recent
developments in web geometry.

For understanding the geometry of hexagonal webs of simple curves
let us first consider $\P^2$ (\FIG{vero}).
In this case, simple curves are exactly the lines and 
$\CG(\P^2)$ is characterized by \THM{coarse}c. Such webs
are of recent interest in combinatorial geometry \citep[Figure~14]{tao}.
Hexagonal webs of lines are characterized
in the following theorem from \citep[1924]{graf1} (see also \citep[Section~1.3, page~24]{bla1}). 
\begin{figure}[!ht]
\centering
\begin{tabular}{c@{\hspace{2cm}}c}
\begin{tikzpicture}
\node[inner sep=0pt] at (0,0) {\fig{4}{4}{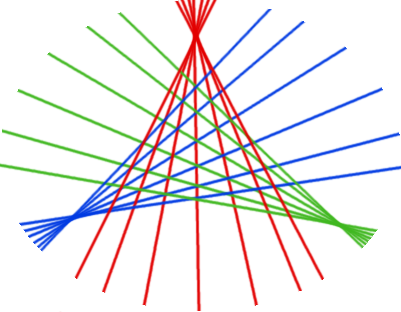}};
\draw[draw=black, fill=red, line width=0.01mm] (-0.04,1.5) circle [radius=1mm];
\draw[draw=black, fill=blue, line width=0.01mm] (-1.3,-0.8) circle [radius=1mm];
\draw[draw=black, fill=black!20!green, line width=0.01mm] (1.4,-0.87) circle [radius=1mm];
\end{tikzpicture}
& 
\begin{tikzpicture}
\node[inner sep=0pt] at (0,0) {\fig{4}{5}{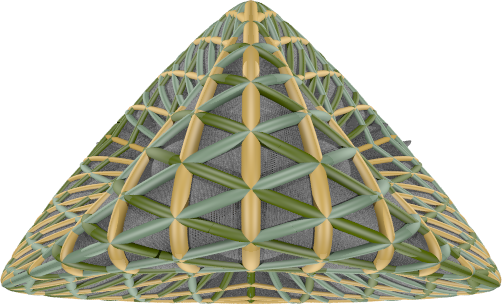}}; 
\end{tikzpicture}
\end{tabular}
\caption{
On the left we see that three pencils of lines form a hexagonal web. 
The three pencils correspond to a reducible cubic consisting of three lines in the dual plane.
The Roman surface on the right is a linear projection to $\P^3$ of a 
Veronese embedding of the planar web into $\P^5$.}
\label{fig:vero}
\end{figure}

\begin{theoremext}[Graf-Sauer, 1924]
\label{thm:graf}
A 3-web of lines in $\P^2$
form a hexagonal web if and only if the lines correspond to points on a cubic curve
in the dual plane $\P^{2*}$ . 
\end{theoremext}

The next surface one might investigate is 
the projective closure $\S^2\subset\P^3$ of the unit sphere $S^2\subset\R^3$. 
The simple curves on $\S^2$ are circles and $\CG(\S^2)$ is characterized by \THM{coarse}b. 
The classification problem for hexagonal webs of circles 
is known as the \df{Blaschke-Bol problem} \citep[\textsection~3, Aufgabe~1, page~31]{bla1}.
The general classification might be a very difficult problem, but some recent progress has been made \cite{shel1, nil1, ako1}.

A \df{Darboux cyclide} is a quartic weak del Pezzo surface in $\S^3$.
For example, the Blum cyclide and ring torus in \FIG{simple} are 
stereographics projections of Darboux cyclides.
Hexagonal 3-webs of circles on Darboux cyclides have recently been classified in \citep[Theorem~17]{pot2}.
We translated this theorem to our setting:

\begin{theoremext}[Pottmann-Shi-Skopenkov, 2012]
Suppose that $X\subset\S^3$ is a Darboux cyclide.
Three vertices in $\CG(X)$ define a hexagonal web if and only if
either the vertices do not share an edge in $\CG(X)$, or
one of the vertices is isolated in $\CG(X)$.
\end{theoremext}

We remark that {\it ``three vertices do not share an edge''} means that 
no two of the three vertices are connected by an edge.
In this article we prove a generalization of the above theorem to any (real algebraic irreducible) surface.

\begin{theorem}
\label{thm:hex}
Let $X\subset\P^n$ be a real surface.
If three vertices in $\CG(X)$ do not share an edge,
then their corresponding three simple families 
are 1-dimensional and form a hexagonal web.
\end{theorem}

We assume that an embedded surface $X\subseteq\P^n$ is not contained in a 
hyperplane section.
The following corollary addresses the Blaschke-Bol problem:

\begin{corollary}
\label{cor:hex}
If $X\subseteq\P^n$ is a real surface that is covered by a hexagonal web
of conics, then $X$ is (a linear projection of) one of the following:
\begin{MenumA}
\item plane with $n=2$.
\item quadric surface with $n=3$.
\item cubic ruled surface with $3\leq n\leq 4$.
\item quartic Veronese surface with $3\leq n\leq 5$.
\item weak del Pezzo surface of degree $d$ \st $3\leq n\leq d\leq 6$.
\end{MenumA}
\end{corollary}

In \FIG{simple} we see hexagonal webs of circles
covering weak del Pezzo surfaces of degree four and six.
\FIG{vero} illustrates \COR{hex}d.

Surface parametrizations whose parameter lines form simple families are 
of interest in geometric modeling:

\begin{theorem}
\label{thm:pmz}
If $X\subset\P^n$ is a real surface \st $\CG(X)$ contains 
two disjoint vertices, then there exists a real birational map
$\P^1\times\P^1\dto X$ of bidegree $(d,d)$, where $d$
is the degree of a simple curve on $X$. Moreover, if 
a real birational map $\P^1\times\P^1\dto X$ 
is of bidegree $(a,b)$, then $a,b\geq d$.
\end{theorem}

Notice that \COR{sphere}, \THM{hex} and \THM{pmz} address the second part of \PRO{1}.

\section{Preliminaries}
\label{sec:pre}

A \df{real variety} $X$ is defined as a complex variety together
with an antiholomorphic involution $\sigma\c X\to X$ \citep[Proposition~I.1.3]{sil1}.
We call $\sigma$ the \df{real structure} of $X$.
Curves, surfaces and projective spaces $\P^n$ are real varieties and maps between them are real, unless explicitly stated otherwise.
In fact, all structures are compatible with the real structure unless explicitly stated otherwise.
We call $X$ \df{$\R$-rational} if there exists a birational map $\P^2\dto X$.

A \df{smooth model} of a surface~$X$ is a birational morphism $Y\to X$ 
from a nonsingular surface~$Y$, such that this morphism 
does not contract complex $(-1)$-curves.
This model exists and is unique up to biregular isomorphisms by~\citep[Theorem~2.16]{kolsing}.
Therefore, we will henceforth refer to a smooth model of a surface as \emph{the} smooth model.

Suppose that $Y\to X$ is the smooth model of a surface $X\subset\P^n$.
A \df{linear normalization} $X_N\subset\P^m$ of $X$ 
is defined as the image of $Y$ via the map associated to 
the complete linear series of hyperplane sections of $X$.
Thus, $m\geq n$, $X$ is a linear projection of $X_N$ and $X_N$ is unique up to $\aut(\P^m)$.

%

\begin{definition}
\label{def:nsl}
The \df{Neron-Severi lattice} $\BN(X)$ of a surface~$X\subset \P^n$ 
with smooth model $\varphi\colon Y\to X$
consists of the following data:
\begin{Menum}
\item
A unimodular lattice $N(X)$ defined by divisor classes on $Y$ modulo numerical equivalence.  

\item
Two distinguished classes $\h,\k\in N(X)$ corresponding to
the \df{class of hyperplane sections} and the \df{canonical class} of~$Y$, \resp.
The class $\h$ is the divisor class of a curve $H\subset Y$ \st $\varphi(H)$
is a hyperplane section of $X$.

\item
A unimodular involution $\sigma_*\c N(X)\to N(X)$ induced 
by the real structure $\sigma\c X\to X$.

\item
A function $h^0\c N(X)\to \Z_{\geq0}$ assigning the dimension of global sections
of the line bundle associated to a class.
\end{Menum}
\newpage
The \df{class} $[C]\in N(X)$ of a complex curve $C\subset X$ 
is defined as the divisor class of 
the 1-dimensional part of its complex preimage $\varphi^{-1}(C)$ minus 
the components that are contracted by the smooth model $\varphi\c Y\to X$.
Notice that $\h$ is the class of a hyperplane section of~$X$.
\END
\end{definition}

\begin{definition}
\label{def:type}
Suppose that $Y\to X$ is the smooth model of a surface $X$.
We call a complex birational morphism $\pi\c Y\to Y'$
a \df{sequential blowup}
if there exists complex blowups $\pi_i\c Y_{i+1}\to Y_i$
with complex centers $p_i\in Y_i$
\st $Y=Y_{r+1}$, $Y'=Y_1$ and
$\pi=\pi_1\circ\ldots\circ\pi_r$.
If $p_{i+1}\in Y_{i+1}$ lies on the complex $(-1)$-curve that contracts to~$p_i\in Y_i$,
then we say that $p_i$ is \df{infinitely near} to $p_{i+1}$.
We call $(p_1,\ldots,p_r)$ a \df{center} of $\pi$
and we refer to $(c_1,\ldots,c_r)$ as \df{exceptional classes} of $\pi$
if $c_i$ is the pullback along~$\pi_{i+1}\circ\ldots\circ\pi_r$ of the class of
the complex $(-1)$-curve that contracts to~$p_i$ for all $1\leq i\leq r$.

Let $\Gamma_1:=\md{\e_0,\e_1,\ldots,\e_r}_\Z$
generate a unimodular lattice \st
$\e_0^2=1$, 
$\e_i^2=-1$ 
and
$\e_0\cdot\e_j=\e_i\cdot\e_j=0$ for $1\leq i<j\leq r$.
We call $\Gamma_1$ a \df{type 1 basis}
for $N(X)$ if $N(X)\cong\Gamma_1$ and there exists a sequential blowup
$\pi\c Y\to \P^2$ with exceptional classes $(\e_1,\ldots,\e_r)$
\st $\e_0$ is the pullback along $\pi$ of the class of a line in~$\P^2$.

Let $\Gamma_2:=\md{\l_0,\l_1, \p_1,\ldots,\p_r}_\Z$
generate a unimodular lattice \st
$\l_0\cdot\l_1=1$,
$\p_i^2=-1$ and
$\l_0^2=\l_1^2=\l_0\cdot\p_i=\l_1\cdot\p_i=\p_i\cdot\p_j=0$ for $1\leq i<j\leq r$.
We call $\Gamma_2$ a \df{type 2 basis} 
for $N(X)$ if $N(X)\cong\Gamma_2$ and there exists a sequential blowup
$\pi\c Y\to \P^1\times\P^1$ with exceptional classes $(\p_1,\ldots,\p_r)$
\st $\l_0$ and $\l_1$ are the pullbacks along $\pi$ of the classes of 
the complex fibers of 
the projections of~$\P^1\times\P^1$ 
to its first and second factor, \resp.

Suppose that $\h\in N(X)$ is the class of hyperplane sections of~$X$.
We call a type~1 basis \df{real} if
$\sigma_*(\e_0)=\e_0$ and
$\sigma_*(\{\e_1,\ldots,\e_r\})=\{\e_1,\ldots,\e_r\}$, 
and we require that
if $\h\cdot \e_i\geq \h\cdot \e_j$ for $0<i,j\leq r$, then $i\leq j$.
We call a type~2 basis \df{real} if 
$\sigma_*(\{\l_0,\l_1\})=\{\l_0,\l_1\}$
and
$\sigma_*(\{\p_1,\ldots,\p_r\})=\{\p_1,\ldots,\p_r\}$,
and we require that
if $\h\cdot \p_i\geq \h\cdot \p_j$ for $0<i,j\leq r$, then $i\leq j$.
Bases with \emph{real} types will only be considered in \SEC{simple}.
\END
\end{definition}

\begin{definition}
\label{def:fam}
Suppose that $X\subset \P^n$ is a surface, 
$B$ a smooth variety and $F\subset X\times B$ a divisor.
We call $F$ a \df{family of curves} of $X$, or \df{family} for short,
if the second projection $\pi_2\c F\to B$ is dominant. 
A \df{member} of $F$, 
corresponding to $b\in B$, is defined as the curve $F_b:=(\pi_1\circ\pi_2^{-1})(b)\subset X$.
If $C, C'\subset X$ are members of the same family, then $[C]=[C']$.
The \df{class of the family} $F$ is defined as the class of any of its members
and is denoted by $[F]$.
\begin{Mlist}
\item
We call $F$ \df{covering} if the first projection $\pi_1\c F\to X$ is dominant. 

\item
We call $F$ \df{rational} if a general member of $F$ has geometric genus $0$.

\item
The \df{degree} of $F$ is defined as the degree of any member \wrt the embedding $X\subset\P^n$.
Equivalently, the degree of $F$ is $\h\cdot [F]$, where $\h\in N(X)$ denotes the class of hyperplane sections. 

 
\item 
We call $F$ \df{minimal} if $F$ is a rational covering family and of minimal degree 
\wrt all rational covering families of $X$.

\item 
We call $F$ \df{complete} if there exists a curve $C\subset X$
\st the set $\set{C'\subset X}{ [C']=[C] }$ defines exactly the set of members of $F$.
In other words, $F$ forms a complete linear series.

\item 
We call $F$ \df{simple} if $F$ is both minimal and complete.

\item
The \df{dimension} of $F$ is defined as $\dim B$. 
If $F$ is complete, then $\dim F = h^0([F])-1$. 
\END

\end{Mlist}
\end{definition}

\begin{definition}
\label{def:wdp}
We call a surface $X\subset\P^n$ a \df{weak del Pezzo surface} 
if $-\k$ is nef and big and if $\h=-\alpha\k$ for some $\alpha\in \Q_{>0}$.
If $-\k$ is also ample, then we call $X$ a (non-weak) \df{del Pezzo surface}.
The \df{canonical degree} of $X$ is defined as $\k^2$.
\END
\end{definition}

\begin{remark}
Notice that we call a surface~$X$ 
with smooth model~$Y\to X$
a (weak) del Pezzo surface 
\Iff 
$Y$ is a (weak) del Pezzo surface as defined in \citep[Definition~8.1.12 and Definition~8.1.18]{dol1}.
Thus any degree preserving linear projection of the anticanonical model of 
a (weak) del Pezzo surface is in our terminology again a (weak) del Pezzo surface.
\END
\end{remark}

\begin{remark}
\label{rmk:h}
The constant $\alpha$ in \DEF{wdp} is characterized in
\citep[Theorem~8.3.2 and Section~8.4.1]{dol1}:
if $1\leq \k^2\leq 2$, then $\alpha\in\Z_{\geq 4-\k^2}$,
if $3\leq \k^2\leq 7$, then $\alpha\in\Z_{>0}$ and
if $8\leq \k^2\leq 9$, then $(\k^2-6)\alpha \in \Z_{>0}$.
\END
\end{remark}

\begin{definition}
\label{def:attr}
Suppose $X\subset\P^n$ is a surface.
We call a class $c\in N(X)$ \df{indecomposable} if $h^0(c)>0$
and there do not exists nonzero $a,b\in N(X)$ \st $c=a+b$ with both $h^0(a)>0$ and $h^0(b)>0$.
The \df{incidence diagram} of a subset $U$ of a Neron-Severi lattice or inner product space
is defined as a labeled graph 
\[
\FD(U):=(U,~\set{(r,s, r\cdot s)}{ r,~ s\in U,~ r\cdot s\neq 0}).  
\]
Throughout this article we will consider subsets of $N(X)$ as stated in \TAB{sub}.
We call $W\subseteq B(X)$ a \df{component}
if its elements define the vertices of an irreducible component of 
the incidence diagram $\FD(B(X))$. 
\END
\end{definition}

\begin{table}[p]
\setstretch{1.2}
\caption{Distinguished subsets of the Neron-Severi lattice $N(X)$ of a weak del Pezzo surface $X$.
See \PRP{class}, \LEM{G} and \PRP{sub}.}
\label{tab:sub}
\centering
\begin{tabular}{l@{~}c@{~}l}
$R(X)$       & $:=$ & $\set{ c\in N(X) }{ \k\cdot c=0 \text{ and } c^2=-2 }$                              \\
$T(X)$       & $:=$ & $\set{ c\in R(X)}{ h^0(c)>0 \text{ or } h^0(-c)>0 }$                     \\
$B(X)$       & $:=$ & $\set{ c\in T(X)}{ c \text{ is indecomposable }}$                         \\
$S(X)$       & $:=$ & $\set{ c\in R(X) }{ \sigma_*(c)=-c }$                                     \\
$A(X)$       & $:=$ & $\text{linear independent subset } (b_i)_i\subset S(X)\text{ such that }$  \\
             & $  $ & $S(X)\cap\set{ \pm\Sigma\alpha_ib_i  }{ \alpha_i\in\Z_{\geq0}}=S(X) \text{ (if it exists)}$ \\
$E(X)$       & $:=$ & $\set{ c\in N(X)}{ \k\cdot c=-1 \text{ and } c^2=-1 }$                               \\
$F(X)      $ & $:=$ & $\set{ c\in N(X) }{ \k\cdot c=-2 \text{ and } c^2=0  }$                  \\
$E_\star(X)$ & $:=$ & $\set{ c\in E(X) }{ c\cdot r\geq 0 \text{ for all } r\in B(X) }$          \\
$F_\star(X)$ & $:=$ & $\set{ c\in F(X) }{ c\cdot r\geq 0 \text{ for all } r\in B(X) }$          \\
$E_\R(X)$ & $:=$ & $\set{c\in E_\star(X)}{ \sigma_*(c)=c }$                                  \\
$F_\R(X)$ & $:=$ & $\set{ c\in F_\star(X) }{ \sigma_*(c)=c }$                                \\
$G(X)      $ & $:=$ & $\set{ c\in N(X) }{ c \text{ is the class of a simple family of } X }$    \\
\end{tabular}
\end{table}

\begin{table}[p]
\caption{Classes in $R(X)$ \wrt a type 1 basis and up to permutation of~$(\e_i)_{i>0}$.
See \PRP{coord}.
}
\label{tab:B}
\centering
\footnotesize
\vspace{-8mm}
\begin{gather*}
\pm(\e_1-\e_2),\qquad
\pm(\e_0-\e_1-\e_2-\e_3),\qquad
\pm(2\e_0-\e_1-\e_2-\e_3-\e_4-\e_5-\e_6),\\
\pm(3\e_0-2\e_1-\e_2-\e_3-\e_4-\e_5-\e_6-\e_7-\e_8).
\end{gather*}
\vspace{-8mm}
\end{table}

\begin{table}[p]
\caption{Classes in $E(X)$ \wrt a type 1 basis and up to permutation of~$(\e_i)_{i>0}$.
See \PRP{coord}.
}
\label{tab:E}
\centering
\footnotesize
\vspace{-5mm}
\begin{tabular}{r@{ }c@{ }r@{ }c@{ }r@{ }c@{ }r@{ }c@{ }r@{ }c@{ }r@{ }c@{ }r@{ }c@{ }r@{ }c@{ }r}
        &     & $ \e_1$ &     &         &     &         &     &         &     &         &     &         &     &         &     &         \\
$ \e_0$ & $-$ & $ \e_1$ & $-$ & $ \e_2$ &     &         &     &         &     &         &     &         &     &         &     &         \\
$2\e_0$ & $-$ & $ \e_1$ & $-$ & $ \e_2$ & $-$ & $ \e_3$ & $-$ & $ \e_4$ & $-$ & $ \e_5$ &     &         &     &         &     &         \\
$3\e_0$ & $-$ & $2\e_1$ & $-$ & $ \e_2$ & $-$ & $ \e_3$ & $-$ & $ \e_4$ & $-$ & $ \e_5$ & $-$ & $ \e_6$ & $-$ & $ \e_7$,&     &         \\
$4\e_0$ & $-$ & $2\e_1$ & $-$ & $2\e_2$ & $-$ & $2\e_3$ & $-$ & $ \e_4$ & $-$ & $ \e_5$ & $-$ & $ \e_6$ & $-$ & $ \e_7$ & $-$ & $ \e_8$,\\
$5\e_0$ & $-$ & $2\e_1$ & $-$ & $2\e_2$ & $-$ & $2\e_3$ & $-$ & $2\e_4$ & $-$ & $2\e_5$ & $-$ & $2\e_6$ & $-$ & $ \e_7$ & $-$ & $ \e_8$,\\
$6\e_0$ & $-$ & $3\e_1$ & $-$ & $2\e_2$ & $-$ & $2\e_3$ & $-$ & $2\e_4$ & $-$ & $2\e_5$ & $-$ & $2\e_6$ & $-$ & $2\e_7$ & $-$ & $2\e_8$.\\
\end{tabular}
\end{table}

\begin{table}[p]
\caption{Classes in $F(X)$ \wrt a type 1 basis and up to permutation of~$(\e_i)_{i>0}$.
See \PRP{coord}.
}
\label{tab:F}
\centering
\footnotesize
\vspace{-5mm}
\begin{tabular}{r@{ }c@{ }r@{ }c@{ }r@{ }c@{ }r@{ }c@{ }r@{ }c@{ }r@{ }c@{ }r@{ }c@{ }r@{ }c@{ }r}
$  \e_0$ & $-$ & $ \e_1$,&     &        &     &        &     &        &     &        &     &        &     &        &     &        \\
$ 2\e_0$ & $-$ & $ \e_1$ & $-$ & $ \e_2$ & $-$ & $ \e_3$ & $-$ & $ \e_4$,&     &        &     &        &     &        &     &        \\
$ 3\e_0$ & $-$ & $2\e_1$ & $-$ & $ \e_2$ & $-$ & $ \e_3$ & $-$ & $ \e_4$ & $-$ & $ \e_5$ & $-$ & $ \e_6$,&     &        &     &        \\
$ 4\e_0$ & $-$ & $2\e_1$ & $-$ & $2\e_2$ & $-$ & $2\e_3$ & $-$ & $ \e_4$ & $-$ & $ \e_5$ & $-$ & $ \e_6$ & $-$ & $ \e_7$,&     &        \\
$ 5\e_0$ & $-$ & $2\e_1$ & $-$ & $2\e_2$ & $-$ & $2\e_3$ & $-$ & $2\e_4$ & $-$ & $2\e_5$ & $-$ & $2\e_6$ & $-$ & $ \e_7$,&     &        \\
$ 4\e_0$ & $-$ & $3\e_1$ & $-$ & $ \e_2$ & $-$ & $ \e_3$ & $-$ & $ \e_4$ & $-$ & $ \e_5$ & $-$ & $ \e_6$ & $-$ & $ \e_7$ & $-$ & $ \e_8$,\\
$ 5\e_0$ & $-$ & $3\e_1$ & $-$ & $2\e_2$ & $-$ & $2\e_3$ & $-$ & $2\e_4$ & $-$ & $ \e_5$ & $-$ & $ \e_6$ & $-$ & $ \e_7$ & $-$ & $ \e_8$,\\
$ 6\e_0$ & $-$ & $3\e_1$ & $-$ & $3\e_2$ & $-$ & $2\e_3$ & $-$ & $2\e_4$ & $-$ & $2\e_5$ & $-$ & $2\e_6$ & $-$ & $ \e_7$ & $-$ & $ \e_8$,\\
$ 7\e_0$ & $-$ & $3\e_1$ & $-$ & $3\e_2$ & $-$ & $3\e_3$ & $-$ & $3\e_4$ & $-$ & $2\e_5$ & $-$ & $2\e_6$ & $-$ & $2\e_7$ & $-$ & $ \e_8$,\\
$ 7\e_0$ & $-$ & $4\e_1$ & $-$ & $3\e_2$ & $-$ & $2\e_3$ & $-$ & $2\e_4$ & $-$ & $2\e_5$ & $-$ & $2\e_6$ & $-$ & $2\e_7$ & $-$ & $2\e_8$,\\
$ 8\e_0$ & $-$ & $3\e_1$ & $-$ & $3\e_2$ & $-$ & $3\e_3$ & $-$ & $3\e_4$ & $-$ & $3\e_5$ & $-$ & $3\e_6$ & $-$ & $3\e_7$ & $-$ & $ \e_8$,\\
$ 8\e_0$ & $-$ & $4\e_1$ & $-$ & $3\e_2$ & $-$ & $3\e_3$ & $-$ & $3\e_4$ & $-$ & $3\e_5$ & $-$ & $2\e_6$ & $-$ & $2\e_7$ & $-$ & $2\e_8$,\\
$ 9\e_0$ & $-$ & $4\e_1$ & $-$ & $4\e_2$ & $-$ & $3\e_3$ & $-$ & $3\e_4$ & $-$ & $3\e_5$ & $-$ & $3\e_6$ & $-$ & $3\e_7$ & $-$ & $2\e_8$,\\
$10\e_0$ & $-$ & $4\e_1$ & $-$ & $4\e_2$ & $-$ & $4\e_3$ & $-$ & $4\e_4$ & $-$ & $3\e_5$ & $-$ & $3\e_6$ & $-$ & $3\e_7$ & $-$ & $3\e_8$,\\
$11\e_0$ & $-$ & $4\e_1$ & $-$ & $4\e_2$ & $-$ & $4\e_3$ & $-$ & $4\e_4$ & $-$ & $4\e_5$ & $-$ & $4\e_6$ & $-$ & $4\e_7$ & $-$ & $3\e_8$.\\
\end{tabular}
\end{table}

\begin{notation}
\label{ntn:RR}
We will use the abbreviations 
(RR), 
(SD), 
(KV), 
(AF) and 
(HI)
for 
Riemann-Roch theorem,
Serre duality,
Kawamata-Viehweg vanishing theorem,
arithmetic genus formula and 
Hodge index theorem, \resp.
Recall that if $X$ is a rational surface, 
then (RR) states that $h^0(c)-h^1(c)+h^2(c)=\frac{1}{2}(c^2-c\cdot\k)+1$
for $c\in N(X)$, where $\k$ denotes the canonical class.
Recall that (SD) states that $h^2(c)=h^0(\k-c)$, 
thus if $-\k$ is nef and $-\k\cdot(\k-c)<0$, then $h^2(c)=0$.
If $c-\k$ is nef and big, then $h^1(c)=h^2(c)=0$ by (KV).
With (AF) we mean the formula $p_a(c)=\frac{1}{2}(c^2+\k\cdot c)+1$.
Thus if $c$ is the class of a irreducible and reduced complex curve, 
then $c^2+\k\cdot c\geq -2$ and $c^2+\k\cdot c$ is even.
Recall that if $b^2>0$ and $b\cdot c=0$ for some $b,c\in N(X)$ \st $c\neq 0$, 
then (HI) states that $c^2<0$.
Notice that if 
$Y\to X$ is the smooth model of a surface $X\subset\P^n$,
then we do intersection theory on $Y$
and not on $X$ itself.
\END
\end{notation}

\section{Divisor classes and weak del Pezzo surfaces}
\label{sec:class}

Suppose that $X\subset\P^n$ is a weak del Pezzo surface.
In this section we describe a correspondence
between classes in $N(X)$ and the following geometric aspects on this surface:
isolated singularities, unmovable curves and simple families. 
We start by considering the subsets $R(X), E(X), F(X)\subset N(X)$
as defined in \TAB{sub}.

\begin{proposition}
\label{prp:coord}
Suppose that $X$ is a weak del Pezzo surface with canonical class $\k$ \st $1\leq \k^2\leq 7$.
\begin{Mlist}
\mclaim{a} 
There exists a (non-real) type 1 basis
$\md{\e_0,\e_1,\ldots,\e_r}_\Z$ for $N(X)$ 
\st $\k=-3\e_0+\e_1+\ldots+\e_r$ and $r=9-\k^2$.

\mclaim{b} 
If $c\in N(X)$ \st $-\k\cdot c=c^2=1$ or $c\in E(X)\cup F(X)$, then $h^0(c)>0$.

\mclaim{c}
Up to permutation of $(\e_i)_{i>0}$, an element of $R(X)$, $E(X)$ and $F(X)$ is \wrt a type 1 basis in 
\TAB{B}, \TAB{E} and \TAB{F}, \resp.
If $c\in N(X)$ \st $-\k\cdot c=c^2=1$, then $c=-\k$.
\end{Mlist}
\end{proposition}

\newpage
\begin{proof} 
Assertion \refmclaim{a} follows from \citep[Theorem~8.1.15]{dol1}. 

\refmclaim{b}
Recall \NTN{RR}.
If $q=(c\cdot\e_0)\,\k+3\,c$, then it follows from \refmclaim{a} that $q\cdot \e_0=0$ 
and thus $q^2<0$ by (HI) so that $(c\cdot\e_0)(-\k\cdot c)>\frac{9}{6}\,c^2$.
If $c\cdot\e_0>-3$, then 
$(\k-c)\cdot \e_0<0$ and thus $h^2(c)=0$ by (SD) and $\e_0$ being nef.
For each of the choices for $(-\k\cdot c, c^2)$ 
we have that $c\cdot\e_0>-3$ and $c^2-\k\cdot c\geq 0$. 
It follows that $h^0(c)>0$ by (RR).

\refmclaim{c}
For a given class $d=d_0\,\e_0+\ldots+d_r\,\e_r$ in $N(X)$ and $\alpha,\beta\in\Z$
we describe a procedure that computes
$\set{c\in N(X)}{ d\cdot c=\alpha,~ c^2=\beta,~ c\cdot \e_i>0 \text{ for } 0\leq i\leq r}$.
We use the notation $c=c_0\,\e_0+\ldots+c_r\,\e_r$ for $c\in N(X)$.
It follows from the Cauchy-Schwarz inequality that
\[
(d_0\,c_0-\alpha)^2=
\left(\Sum{i>0}{}d_i\,c_i\right)^2 
\leq
\left(\Sum{i>0}{}d_i^2\right)\left(\Sum{i>0}{}c_i^2\right)
=
(d_0^2 - d^2)(c_0^2 - \beta).  
\]
Thus we require that $f(c_0)\leq 0$ where 
$f(t):=(d_0\,t-\alpha)^2-(d_0^2 - d^2)(t^2 - \beta)$.

We start the procedure with $c_0:=1$.
We compute all $(c_i)_{i>0}\in\Z^r_{\geq0}$ \st 
$d_0\,c_0-\alpha=\sum_{i>0}d_i\,c_i$ by going 
through integer partitions of $d_0\,c_0-\alpha$.
If $(d\cdot c,c^2)=(\alpha,\beta)$, then we add $c$ to our list. 
We increase $c_0$ with one and repeat the above steps.
Notice that if $f(c_0)> f(c_0-1)$ and $f(c_0)>0$, then $f(t)>0$ for all $t\geq c_0$.   
Thus in this case we stop as we obtained all possible elements
of the required set. 

It follows from a) and \DEF{type} that 
if $h^0(c)>0$ for $c\in N(X)$, then either $c\cdot\e_i>0$ for $0\leq i\leq r$
or $c\in\set{\e_i, \e_i-\e_j}{1\leq i,j\leq r}$.
Assertion \refmclaim{c} now follows from applying \refmclaim{b} 
and computing the sets $R(X)$, $E(X)$, $F(X)$ and 
$\set{c\in N(X)}{ -\k\cdot c=c^2=1 }$ using the 
above procedure. 
We verify that the latter set is equal to $\{-\k\}$ as asserted.
For $R(X)$ and $E(X)$ see alternatively \citep[Proposition~8.2.7 and Proposition~8.2.19]{dol1}.
\end{proof}

\begin{lemma}
\label{lem:div}
Suppose that $X$ is a weak del Pezzo surface with canonical class $\k$ \st $1\leq \k^2\leq 7$.
If $c\in E(X)\cup F(X)$, then $c=u+b$ \st $b:=\Sigma_{i\in I}b_i$ with $b_i\in B(X)$ for all $i\in I$
and either
\begin{Mlist}
\item[i)] $u\in E_{\star}(X)$, $u$ is indecomposable and $h^0(u)=1$,
\item[ii)] $u\in F_{\star}(X)$, $u$ is the movable part of $c$ and $h^0(u)=2$, or
\item[iii)] $\k^2\leq 2$ and $u,c\notin E_{\star}(X)\cup F_{\star}(X)$.
\end{Mlist}
Moreover, if $c\in E_\star(X)\cup F_\star(X)$, then $I=\emptyset$.
\end{lemma}

\newpage
\begin{proof} We use \NTN{RR}. Suppose that $v\in N(X)$ \st $h^0(v)>0$
and let $v=v_1+\ldots+v_t$ a decomposition 
into a sum of $t\geq 1$ indecomposable classes
\st $-\k\cdot v_1\geq -\k\cdot v_i$ for all $i$.

{\it Claim 1:} If $h^0(v)=1$ and $0\leq -\k\cdot v\leq 1$,
then $v_i\in E_\star(X)\cup B(X)$ for all $1\leq i\leq t$.
\\
Since $-\k$ is nef, we find that $0\leq -k\cdot v_1\leq 1$ and $-k\cdot v_i=0$ for all $i>1$.
It follows from (HI) and $\k^2>0$ that $v_i^2\leq 0$ and thus $v_i\in B(X)$ for $i>1$ by (AF).
Suppose that $-k\cdot v_1=1$ so that $h^2(v_1)=0$ by (SD)
and thus $v_1^2<0$ by (RR) so that $v_1^2=-1$ by (AF). 
Since $v_1\cdot v_i\geq 0$ for all $i>1$ 
it follows that $v_1\in E_\star(X)$. If $-k\cdot v_1=0$, then $v_1\in B(X)$ and thus claim~1 holds.

{\it Claim 2:} If $-k\cdot v=1$, then either $v^2=-1$, or $v=-\k$ with $\k^2=1$.
\\
By \RMK{h} we may assume \Wlog that $\h=-3\k$ so that $v$ is the class of a complex 
irreducible cubic in~$X$. 
Thus $0\leq p_a(v)\leq 1$ by \citep[Example~IV.6.4.2]{har1} so that $v^2=\pm 1$ by (AF). 
Claim~2 now follows from \PRP{coord}c.

Since $c\in E(X)\cup F(X)$,
it follows from \PRP{coord}b that $h^0(c)>0$.
Let $c=m+f$, where $m$ and $f$ are the movable and fixed part of $c$, \resp. 

Suppose that $c\in E(X)$.
If $m=0$, then (i) holds by claim~1.
If $m\neq 0$, then we know from (HI), $\k^2>0$ and $m^2\geq 0$ that $\k\cdot m\neq 0$.
Therefore, $-\k\cdot m=1$ as $-\k$ is nef. 
It follows from claim~2 that $m^2=1$ and $m=-\k$ so that 
(iii) holds.

Suppose that $c\in F(X)$.
Notice that $m\neq 0$, since otherwise $h^0(f)\geq 2$ by (SD) and (RR).
Since $-\k\cdot m-\k\cdot f=2$ and $-\k$ is nef, we find that $1\leq -\k\cdot m\leq 2$.

Suppose that $-\k\cdot m=1$.
It follows from claim~1 with $v=f$ and claim~2 with $v=m$ 
that $c=-\k+e+b$ \st $e\in E_*(X)$ and $\k^2=1$.
Thus $u=-\k+e$ so that $c\cdot u=2+b\cdot e>0$.
Since $c^2=c\cdot u+c\cdot b=0$ it follows that $c\cdot b<0$
and thus (iii) holds as $c\cdot b_i<0$ for some $i\in I$.

Suppose that $-\k\cdot m=2$ so that $-k\cdot f=0$.
It follows from claim~1 that $u=m$ and $b=f$.
Recall from \RMK{h} that if $\k^2\geq 3$, then \Wlog $\h=-\k$ so that $m$
is the class of conic and thus $u^2=0$ by (AF).
By (RR) and (KV) we find that $h^0(u)=\frac{1}{2}u^2+2$.
Thus if $\k^2\geq 3$ or $u^2=0$, then (ii) holds.
Suppose that $\k^2\leq 2$ and $u^2>0$. 
It follows from (HI) and $c^2=0$ that $u\cdot c>0$.
Since $c^2=u\cdot c+b\cdot c=0$ we find that $c\cdot b<0$ so that (iii) holds
as $c\cdot b_i<0$ for some $i\in I$.

\newpage
Suppose by contradiction that $c\in E_\star(X)\cup F_\star(X)$ and $I\neq \emptyset$.
We are in (i) or (ii) so that $c=u+b$ and $u\in E_\star(X)\cup F_\star(X)$. 
It follows from $\k^2>0$ and (HI) that $b^2<0$.
Since $c^2=u^2+b\cdot u+c\cdot b=u^2$ we find that $c\cdot b=-u\cdot b$.
We have $c\cdot b=u\cdot b+b^2$ 
and it follows from $\k^2>0$ and (HI) that $b^2<0$.
We arrived at a contradiction, since $c\cdot b<0$ and thus $c\cdot b_i<0$ 
for some $i\in I$.
\end{proof}

\begin{lemma}
\label{lem:blowdown}
Suppose that $Y\to X$ is the smooth model 
of a weak del Pezzo surface~$X$ with canonical class $\k$.
If $e\in E_\star(Y)$, 
then $e$ is the class of a complex $(-1)$-curve
that can be contracted \st the resulting smooth complex surface~$Y'$ 
is a complex weak del Pezzo surface of canonical degree $2\leq \k^2+1\leq 9$
and $N(Y')$ is isomorphic to the sublattice $\set{c\in N(Y)}{ c\cdot e=0 }$. 
If $\k^2=9$, then $Y\cong \P^2$ and $N(Y)=\md{\e_0}_\Z$.
If $\k^2< 8$, then $Y$ is complex isomorphic to the blowup of $\P^2$ in $9-\k^2$ points.
\end{lemma}

\begin{proof}
See \citep[Proposition~8.1.23 and Theorem~8.1.15]{dol1}. 
\end{proof}

The following lemma will be used only later in \LEM{pen},
but now is the right time to prove it.

\begin{lemma}
\label{lem:orth}
Suppose that $X$ is a weak del Pezzo surface with canonical class $\k$.
\begin{Mlist}
\mclaim{a} If $\k^2\leq 7$ and $f\in F_\star(X)$, 
then there exists $e\in E_\star(X)$ \st $f\cdot e=0$.

\mclaim{b} If $\k^2\leq 7$ and $f_1,f_2\in F_\star(X)$ \st $f_1\cdot f_2=1$,
then there exists $e\in E_\star(X)$ \st $f_1\cdot e= f_2\cdot e=0$.

\mclaim{c} If $\k^2\leq 5$ and $f_1,f_2,f_3\in F_\star(X)$ \st $f_1\cdot f_2=f_1\cdot f_3=f_2\cdot f_3=1$,
then there exists $e\in E_\star(X)$ \st $f_1\cdot e= f_2\cdot e=f_3\cdot e=0$.
\end{Mlist}
\end{lemma}

\begin{proof}
Suppose that $g\in F_\star(X)$ and that there exists $e\in E(X)$ \st $g\cdot e=0$.
By \LEM{div} we find that $e=u+b$ where $u\in E_\star(X)$ and $b$ is a sum of classes in $B(X)$. 
We have $g\cdot b\geq 0$ by assumption and $g\cdot u\geq 0$ by \LEM{div}.
It follows that there exists $u\in E_\star(X)$ \st $g\cdot u=0$. 
Therefore, we may assume \Wlog that $B(X)=\emptyset$ in the proofs of 
the assertions \refmclaim{a}, \refmclaim{b} and \refmclaim{c}.
In other words, we assume \Wlog that $E_\star(X)=E(X)$.

Suppose that $c\in E(X)\cup F(X)$.
By \PRP{coord},
$c$ is, up to permutation of the $(\e_j)_j$, defined by the $\lambda$-th row in \TAB{E} or \TAB{F}.
We use in this case the notation $c\to \lambda$.
For example, $\e_1\to 1$, $\e_8\to 1$, $\e_0-\e_2\to 1$ and $2\e_0-\e_4-\e_5-\e_6-\e_7\to 2$. 

\refmclaim{a}
First we suppose that $\k^2=1$.
Suppose by contradiction that there does not exist $e\in E(X)$ \st $e\cdot f=0$.
Notice that $f\to \lambda$ for some row $1\leq \lambda\leq 15$.
We find that          $\lambda\geq 4$,                  otherwise $e \cdot f=0$ for some $e\in E(X)$ \st $e\to 1$.
Moreover,             $\lambda\in\{9,11,12,13,14,15\}$, otherwise $e \cdot f=0$ for some $e\in E(X)$ \st $e\to 2$.
Next, we observe that $\lambda\in\{  11,   13,14,15\}$, otherwise $e \cdot f=0$ for some $e\in E(X)$ \st $e\to 3$. 
By the same argument we have $r\in\{15\}$, otherwise $e \cdot f=0$ for some $e\in E(X)$ \st $e\to 4$ or $e\to 5$.           
We arrived at a contradiction as $e\cdot f=0$ for some $e\in E(X)$ \st $e\to 6$.
The cases for $2\leq \k^2\leq 7$ can be shown analogously.

\refmclaim{b}
It follows from \refmclaim{a} that $f_1\cdot e=0$ for some $e\in E(X)$.
If $f_2\cdot e\neq 0$, then $f_2\cdot e>0$ by \LEM{div}. 
We find that $f_1-e$ is in $E(X)$ and orthogonal to $f_1$.
Again by \LEM{div} we have that $f_2\cdot(f_1-e)\geq 0$ and thus $f_2\cdot(f_1-e)=0$
so that we concluded the proof.

\refmclaim{c}
We know from \refmclaim{b} that there exists $e\in E_\star(X)$
\st $f_1\cdot e=f_2\cdot e=0$. We subsequently apply \LEM{blowdown} and \refmclaim{b} so that 
we obtain a complex birational morphism $\varphi\c X\to X'$
where $X'$ is a weak del Pezzo surface of canonical degree $7$ and 
$\varphi_*f_1,\varphi_*f_2 \in F(X)$.
It follows from \PRP{coord}c that we may assume 
\Wlog that $\varphi_*f_1=\e_0-\e_1$ and $\varphi_*f_2=\e_0-\e_2$.
Thus it follows from \LEM{blowdown} that 
$N(X)\cong \md{\e_0,\ldots,\e_r}_\Z$ with $r=9-\k^2\geq 4$ \st $f_1=\e_0-\e_1$, $f_2=\e_0-\e_2$
and $\e_3,\ldots, \e_r \in E_\star(X)$ are the classes of the 
complex $(-1)$-curves that are contracted by $\varphi$.
It is now a straightforward consequence of \PRP{coord}c 
and $f_1\cdot f_3=f_2\cdot f_3=1$ that either $f_3\to 1$ or $f_3\to 2$.
In both cases there
exists $u\in E(X)$ with either $u\to 1$ or $u\to 2$ \st 
$f_1\cdot u=f_2\cdot u=f_3\cdot u=0$ and thus we concluded the proof.
\end{proof}

\begin{proposition}
\label{prp:class}
Suppose that $X\subset\P^n$ is a weak del Pezzo surface with 
canonical class~$\k$ \st $1\leq \k^2\leq 6$.
Let $\h$ denote the class of hyperplane sections of~$X$,
let $Y\to X$ be the smooth model
and let $X_N$ be a linear normalization of $X$.
\begin{Mlist}
\mclaim{a} 
We have that $W\subset N(X)$ is the set of classes of 
all complex irreducible curves in~$Y$ 
that contract to some complex isolated singularity of~$X_N$ 
\Iff $W$ is a component of $B(X)$.
If $\sigma_*(W)=W$, then the corresponding isolated singularity is real.

\newpage
\mclaim{b}
We have that $c\in E_\star(X)$ \Iff 
$c=[C]$ and $C\subset X$ is an unmovable complex minimal degree rational curve.
If $c\in E_\R(X)$, then the corresponding curve is real. 
If $\h=-\k$, then the unmovable rational curves on $X$ are lines.

\mclaim{c}
We have that $c\in F_\star(X)$ \Iff 
$c$ is the class of a 1-dimensional complex minimal family of~$X$.
If $c\in F_\R(X)$, then the corresponding family is simple.
If $\h=-\k$, then simple curves on $X$ are conics.
\end{Mlist}
\end{proposition}

\begin{proof}
The assertion \refmclaim{a} is a direct consequence of \citep[Proposition~8.1.9, Proposition~8.1.10 and Theorem~8.2.27]{dol1}.
Assertions \refmclaim{b} and \refmclaim{c} are straightforward consequences of \LEM{div}.
\end{proof}

\begin{notation}
\label{ntn:F}
Let $\BF_0$ and $\BF_0'$ be both isomorphic to $\P^1\times\P^1$,
but with $\sigma_*$ acting trivial and non-trivial, \resp.
Let $\BF_1$ denote the blowup of $\P^2$ in one point. 
Let $\BF_2$ denote the blowup of $\P^2$ in two infinitely near points followed by the contraction
of the pullback of the line through these points.
Here $\sigma_*\c N(\BF_i)\to N(\BF_i)$ acts trivially for $1\leq i\leq 2$.
Let $\BG_2$ denote the blowup of $\P^2$ in two complex conjugate points
so that $N(\BG_2)=\md{\e_0,\e_1,\e_2}_\Z$ is a type 1 basis.
\END
\end{notation}

\begin{lemma}
\label{lem:G}
Suppose that $Y\to X$ is the smooth model of 
a weak del Pezzo surface~$X\subset\P^n$ 
with canonical class~$\k$. 
Either one of the following 8 cases holds (we use \NTN{F} and \TAB{sub}):
\begin{MenumA}
\item $\k^2=9$, $Y\cong\P^2$ and $G(X)=\{-\frac{1}{3}\k\}$.

\item $\k^2=8$, $Y\cong\BF_0 $ and $G(X)=F_\R(X)$. 
\item $\k^2=8$, $Y\cong\BF_0'$ and $G(X)=\{-\frac{1}{2}\k\}$.
\item $\k^2=8$, $Y\cong\BF_1 $ and $G(X)=F_\R(X)$.
\item $\k^2=8$, $Y\cong\BF_2 $ and $G(X)=F_\R(X)$.

\item $\k^2=7$, $Y\cong\BG_2$ and $G(X)=\{\e_0\}$.

\item $1\leq\k^2 \leq 7$ and $G(X)=F_\R(X)\neq\emptyset$.

\item $1\leq\k^2 \leq 2$, $G(X)=\emptyset$ and $X$ is not $\R$-rational.
\end{MenumA}
\end{lemma}

\begin{proof}
For $9\geq \k^2\geq 7$, the assertion is a straightforward consequence of 
\citep[Section~8.4.1]{dol1}.

Now, suppose that $X$ is $\R$-rational and $1\leq \k^2\leq 6$.
Thus $Y$ is the blowup of either $\P^2$ in at least three complex points, 
or the blowup of $\P^1\times\P^1$ in at least two complex points. 
It follows from \PRP{coord}c and \LEM{div} that if $Y$ is the blowup of $\P^2$ in at least one real point, 
then \Wlog $\e_0-\e_1\in F_\R(X)$.
Similarly, if $Y$ is the blowup of $\P^2$ in four complex conjugate points, 
then \Wlog $2\e_0-\e_1-\e_2-\e_3-\e_4\in F_\R(X)$.
If $Y$ is the blowup of $\P^1\times\P^1$ in two complex conjugate points, 
then the pullback of a curve of bidegree (1,1), that passes through these points, has class $c\in F(X)$ \st $\sigma_*(c)=c$
so that $F_\R(X)\neq \emptyset$ by \LEM{div}.
Thus if $X$ is $\R$-rational, then $G(X)=F_\R(X)$ by \PRP{class}c.

We established that if $F_\R(X)=\emptyset$ and $\k^2\leq 6$, then $X$ is not $\R$-rational and 
thus $X$ is not covered by a complete rational family (see \DEF{fam}) so that $G(X)=\emptyset$.
It follows from \citep[Theorem~4.6]{sil1} that $1\leq\k^2 \leq 2$.
\end{proof}

\begin{remark}
\label{rmk:G}
Suppose that $X\subset\P^n$ is an algebraic surface.
We recover the simple family graph $\CG(X)$ uniquely from $G(X)$, 
by removing from the incidence diagram $\FD(G(X))$ the self-loops and 
the edges labeled $1$,
and by labeling each vertex $c\in G(X)$ with $h^0(c)-1$.
\END
\end{remark}

\section{Cremona invariant for weak del Pezzo surfaces}

We will introduce an invariant for Neron-Severi lattices
of weak del Pezzo surfaces.

We start by recalling some concepts of root systems and we mainly follow \citep[Chapter 8]{hab1}.
Let $(V,\cdot)$ be a real inner product space.
A \df{root system} is a finite subset $R\subset V$
with the following properties:
\begin{Menum}
\item The set $R$ spans $V$.
\item If $r\in R$ and $\alpha r\in R$ for $\alpha\in\R$, then $\alpha\in\{1,-1\}$.
\item The set $R$ is closed under reflection through the hyperplane perpendicular to any $r\in R$:
$
s-2(s\cdot r/ r\cdot r)r \in R \textrm{ for all } r,s\in R.
$        
\item The projection of $s\in R$ onto the line through $r\in R$ is a half-integral multiple of $r$:
$
2(s\cdot r/ r\cdot r)\in\Z.
$ \END
\end{Menum}
We call $B\subset R$ a \df{root base} of a root system $R$
if $B$ is a basis for $V$ and if every element in $R$
can be expressed as an integral linear combination of elements in 
$B$ with all coefficients either all positive or all negative.

\newpage
The \df{Dynkin diagram} of a root system $R$ 
is defined as the incidence diagram 
of a root base $B\subset R$ (see \DEF{attr}) and is independent of the choice of root base.
In this article, the label of an edge in a Dynkin diagram 
is $1$ if the edge is not a self-loop
and all vertices in our Dynkin diagrams admit a self-loop
labeled $-2$.
Thus we can ignore the labelings and self-loops
so that the Dynkin diagrams in this text will be isomorphic to 
either the empty graph $A_0$ or one of the Dynkin graphs $A_\ell$, $D_\ell$ or $E_\ell$ with $\ell$ vertices.

Root systems $R$ and $R'$ are isomorphic \Iff 
there exists an isomorphism 
$\varphi\c (V,\cdot)\to (V,\cdot)$
\st $\varphi(R)=R'$. 
In \citep[Definition 8.5]{hab1} the notion of isomorphism is slightly weaker, but
equivalent in the setting of this paper.
The Dynkin diagram of a root system determines uniquely its isomorphism class.

A \df{root subsystem} is a subset $S\subset R$ that forms root system 
in the vector space spanned by $S$. Two root subsystems $(S,R)$ and $(S',R')$
are isomorphic \Iff there exists an isomorphism 
of root systems $\varphi\c R\to R'$ \st $\varphi(S)=S'$.
The Dynkin diagram of a root subsystem does in general {\it not} uniquely determine its isomorphism class.

Let us consider the distinguished subsets of $N(X)$ as stated in \TAB{sub}.
We consider the following inner product space 
\[
V_{\k}(X):=(~\set{ c\in N(X) }{ \k\cdot c=0 }\otimes_\Z\R,~ \cdot~).
\]
The following proposition summarizes some known facts from the literature.

\begin{proposition}
\label{prp:sub}
Let $X$ and $X'$ be weak del Pezzo surfaces \st $1\leq \k^2=\k^{'2}\leq 8$,
$\h=-\alpha\k$ and $\h'=-\alpha\k'$ for some fixed $\alpha\in\Q_{>0}$.
\begin{Mlist}
\mclaim{a} $R(X)$ forms a root system in the inner product space $V_{\k}(X)$
\st 
\\[2mm]
\begin{tabular}{r|cccccccc}
$\k^2$ &
$8$ &
$7$ &
$6$ &
$5$ &
$4$ &
$3$ &
$2$ &
$1$ 
\\
Dynkin diagram of $R(X)$ &
$A_1$     &        
$A_1$     &       
$A_1+A_2$ &  
$A_4$     &  
$D_5$     &  
$E_6$     &
$E_7$     &
$E_8$     
\end{tabular}

\mclaim{b}
Subsets $S(X)$ and $T(X)$ form root subsystems of $R(X)$
with root bases $A(X)$ and $B(X)$, \resp.

\mclaim{c} Suppose that $\sigma_*$ is the identity. 
One has $\BN(X)\cong \BN(X')$ \Iff $T(X)\cong T(X')$ as root subsystems.

\mclaim{d} 
Suppose that $B(X)$ is the emptyset.
One has $\BN(X)\cong \BN(X')$ \Iff $S(X)\cong S(X')$ as root subsystems.

\end{Mlist}
\end{proposition}

\begin{proof}
Assertions \refmclaim{a} and \refmclaim{b} follow from
\citep[Proposition~8.2.10 and Proposition~8.2.25]{dol1}
and \citep[Section~2]{wal1} (see also \RMK{A}).
Notice that if $\sigma_*$ is the identity or if $B(X)=\emptyset$, then we can consider $X$ as a complex weak del Pezzo surface
and a (non-weak) del Pezzo surface, \resp.
Assertion \refmclaim{c} is now a direct consequence of \citep[Corollary~8.2.33]{dol1}. 
Assertion \refmclaim{d} is a direct consequence of \citep[Theorem~2.1]{wal1}.
\end{proof}

\begin{remark}
\label{rmk:A}
The matrix defining $\sigma_*\c N(X)\to N(X)$ has eigenvalues $\pm 1$, since it is an involution.
Let $V^\pm$ denote the eigenspace of $\sigma_*$ for eigenvalues $\pm1$, \resp.
Notice that $V^-\subset V_\k(X)$, since
$\sigma_*(\k)=\k$ and $\sigma_*(v)\cdot \sigma_*(\k)=-v\cdot k$ for all $v\in V^-$.
The intersection of a root system with a subspace is a root subsystem
and $V^-\cap R(X)$ generates $V^-$.
Thus $S(X)$ is a root subsystem of $R(X)$ and $A(X)$ is its root base by definition. 
\END
\end{remark}

\begin{lemma}
\label{lem:h0}
Let $X$ be a weak del Pezzo surface.
The function $h^0\c N(X)\to \Z_{\geq0}$ is uniquely determined by $B(X)$.
Thus the Neron-Severi lattice $\BN(X)$ is uniquely determined by $B(X)$ together with data 1, 2 and 3 in \DEF{nsl}.
\end{lemma}

\begin{proof}
Recall \NTN{RR},
see \TAB{sub} for $E(X),B(X)\subset N(X)$ and
let $\k$ denote the canonical class of $X$ \st $-\k$ is nef and big by \DEF{wdp}.
By \PRP{coord} we may assume that 
$\md{\e_0,\e_1,\ldots,\e_r}_\Z$ is a type 1 basis for $N(X)$ 
\st $r=9-\k^2$, $\k=-3\e_0+\e_1+\ldots+\e_r$
and the elements of $E(X)$ are characterized by \TAB{E}.
We assume that we have explicit coordinates for the elements in $B(X)$
in terms of this type 1 basis and
we recover $E_\star(X)$ from $E(X)$ and $B(X)$
by applying \LEM{div}.

{\it Claim 1:} 
If $c\in N(X)$ \st
$h^0(c)>0$ and $c\cdot q<0$
for some indecomposable class $q$,
then $q\in E_\star(X)\cup B(X)$ and $h^0(c)=h^0(c-q)$.
\\
We know from (RR) and (SD) that 
$h^0(q)=1\geq \frac{1}{2}(q^2-\k\cdot q)+1$ so that $-q^2+\k\cdot q\geq 0$.
Therefore, by (AF) and $-\k$ being nef, we find that $0\leq -\k\cdot q\leq 1$ and $q^2\leq 0$.
If $-\k\cdot q=0$, then $q^2<0$ by (HI) and thus $q^2=-2$ by (AF).
If $-\k\cdot q=1$, then $q^2=-1$ by (AF).
It follows from \LEM{div} that $q\in E_\star(X)\cup B(X)$.
Since $q$ is the class of a fixed component of the linear series of $c$
we find that $h^0(c)=h^0(c-q)$.
Hence, we concluded the proof of claim~1.

For given $c\in\md{\e_0,\e_1,\ldots,\e_r}_\Z$
we would like to recover $h^0(c)$ from $E_\star(X)\cup B(X)$.
If $c\cdot q<0$ for some $q\in E_\star(X)\cup B(X)$,
then we apply claim~1 and set $c:=c-q$.
We repeat this step 
until either $c\cdot q\geq 0$ for all $q\in E_\star(X)\cup B(X)$
or $c\cdot e_0<0$.
If $c\cdot e_0<0$, then $h^0(c)=0$.
Otherwise $h^0(c)=\frac{1}{2}(c^2-\k\cdot c)+1$ by (RR) and (KV).
\end{proof}

The pair of Dynkin diagrams $\FD(A(X))$ and $\FD(B(X))$
is not fine enough as an invariant for the Neron-Severi lattice $\BN(X)$
of a weak del Pezzo surface $X$,
since non-isomorphic Neron-Severi lattices might have the same such pair.  
For this purpose we introduce  
a more fine-grained and computable version of the Dynkin diagram.
The \df{Cremona invariant} is defined as
\[
\FC(X):=\FD\bigl(E(X)\cup B(X)\bigr)
\,\bigcup\,
\set{(c,\sigma_*(c),\infty)}{ c\in E(X)\cup B(X)}.
\]
Thus we enhance the incidence diagram of $E(X)\cup B(X)$ with edges (labeled $\infty$) between conjugate classes.
Notice that $\sigma_*(B(X))=B(X)$.
We consider Cremona invariants isomorphic if they are isomorphic as labeled graphs.

\begin{proposition}
\label{prp:cremona}
Suppose that $X$ and $X'$ are weak del Pezzo surfaces
\st $1\leq \k^2,\k'^2\leq 7$, $\h=-\alpha\k$ and $\h'=-\alpha\k'$ for some fixed $\alpha\in\Q_{>0}$.
One has $\BN(X)\cong \BN(X')$
\Iff 
$\FC(X)\cong \FC(X')$.
\end{proposition}

\begin{proof}
The $\Rightarrow$ direction is immediate.
In the remainder of this proof we consider the $\Leftarrow$ direction. 
It is a straightforward consequence of \PRP{coord}[a,c] that 
$N(X)\cong \md{\e_0-\e_1-\e_2, \e_1,\ldots, \e_r}_\Z$, where the generators are a subset of~$E(X)$.
The graph isomorphism restricted to these generators,
uniquely induces an isomorphism $\lambda\c N(X)\to N(X')$ of $\Z$-modules.
It is left to verify that this isomorphism is compatible with 
the remaining data of $\BN(X)$ as it is stated in \DEF{nsl}.
The edges labeled $\infty$ uniquely determines the involution $\sigma_*$
on the generators \st $\lambda\circ\sigma_*=\sigma_*\circ \lambda$.
The canonical class $\k$ of $X$ is the unique element in $N(X)$
\st $\k^2=10-\rnk(N(X))$, $\sigma_*(\k)=\k$ and $\k\cdot e=-1$ for all $e\in E(X)$.
It follows that $\lambda(\k)=\k'$.
It follows from \LEM{h0} and $\lambda(B(X))=B(X')$ that $h^0=h^0\circ\lambda$.
Thus we constructed an isomorphism of Neron-Severi lattices and concluded the proof.
\end{proof}

\section{A classification of weak del Pezzo surfaces}

In this section we prove \THM{nsl} via an algorithm that outputs \TAB{nsl}.
The algorithm is designed so that it is easy to implement in a computer algebra system,
while still terminating in less than two days on a single pentium processor.
The computation of \TAB{nsl} for $d\geq 3$ takes only a couple of minutes.
See \citep[{\tt ns\_lattice}]{ns_lattice} for an implementation.

The \df{set of positive roots} $R^+$ is defined as the set of classes $c\in \md{\e_0,\ldots,\e_8}_\Z$ 
that are listed up to permutation of $(\e_i)_i$ in \TAB{B} 
\st either $c=\e_i-\e_j$ for some $1\leq i<j\leq 8$ or $c\cdot \e_0>0$.

\begin{lemma}
\label{lem:R}
Suppose that $X\subset\P^n$ is a weak del Pezzo surface \st $1\leq\k^2\leq 7$. 
There exists a type 1 basis for $N(X)$ \st $A(X)\subset R^+$.
Moreover, for all unimodular involutions $\sigma_*\c N(X)\to N(X)$, 
\st $A(X)\subset R^+$ is a root base for $S(X)$,
we may assume \Wlog that $B(X)\subset R^+$.
\end{lemma}

\begin{proof}
Let $Y\to X$ be the smooth model and
let $X_N$ be a linear normalization.
By \PRP{coord}a, there exists a sequential blowup $\pi\c Y\to \P^2$
with center $(p_1,\ldots, p_r)$
so that $N(X)$ admits a type 1 basis $\md{\e_0,\ldots,e_r}_\Z$.
We follow the notation of \DEF{type} and let $C_1\subset Y_1=\P^2$
be a curve.
The Zariski closure $C_2\subset Y_2$ of $\pi_1^{-1}(C_1\setminus\{p_1\})$
is called the \df{strict transform} of $C_1$.
The strict transform of $C_1$ via $\pi_1\circ\pi_2$
is defined as the strict transform of~$C_2$ and so on.
Recall from \RMK{A} that the root base~$A(X)$ 
forms a basis of the eigenspace of~$-1$ for the involution~$\sigma_*$.
Since positive roots in~$A(X)$ are send by $\sigma_*$ to negative roots,
we may assume \Wlog that $A(X)\subset R^+$.
Any other choice for~$\sigma_*$, \st $A(X)\subset R^+$ is a root base for $S(X)$,
is equivalent to a choice for a type 1 basis for $N(X)$.
If $X_N\cong Y$, then $B(X)=\emptyset$ by \PRP{class}a and 
the complex blowup centers $p_1,\ldots,p_r$ are general.
Now suppose that $X_N\ncong Y$.
In this case $p_1,\ldots,p_r$ lie in ``almost'' general position 
in the following sense.
If $\e_0-\e_1-\e_2-\e_3\in B(X)$, then $p_1$, $p_2$ and $p_3$ lie on 
strict transforms of a complex line.
If $2\e_0-\e_1-\e_2-\e_3-\e_4-\e_5-\e_6\in B(X)$, then $p_1,\ldots p_6$ lie on 
strict transforms of a complex conic.
If $3\e_0-2\e_1-\e_2-\e_3-\e_4-\e_5-\e_6-\e_7-\e_8\in B(X)$, then $p_1,\ldots p_8$ lie on a 
strict transforms of a complex cubic that has a double point at $p_1$.
If $\e_1-\e_2 \in B(X)$, then $p_2$ is infinitely near to $p_1$.
It follows from \PRP{coord}c that these are, up to permutation of 
$p_1,\ldots,p_r$ all possible geometric conditions.
It is now straightforward to see that we can always choose an index on the points \st 
all the conditions are expressed by elements in $R^+$ and \st $\sigma_*(B(X))=B(X)$.
Therefore, $B(X)\subset R^+$ \Wlog as was left to be shown. 
\end{proof}

\begin{notation}
\label{ntn:C}
Suppose that $X$ is a weak del Pezzo surface with canonical class~$\k$.
In the following five algorithms we will represent 
$N(X)$ as a type~1 basis $\md{\e_0,\ldots,\e_r}_\Z$ with $r=9-\k^2$
and $R^+$ is short notation for~$R^+\cap N(X)$. 
Suppose that $B\subset R^+$ represents a root base $B(X)$ and 
that matrix $M \in \Z^{(r+1)\times(r+1)}$ represents $\sigma_*\c N(X)\to N(X)$.
We denote by $\FC(B,M)$ the Cremona invariant $\FC(X)$.
Notice that we know $E(X)\subset \md{\e_0,\ldots,\e_r}_\Z$ from \TAB{E}
and thus we can uniquely recover $\FC(X)$ from $B$ and $M$. 
If $M$ is the identity matrix, then we denote $\FC(X)$ by $\FC(B)$. 
\END
\end{notation}

\begin{algorithm}
({\tt is\_root\_base})
\label{alg:basis}
\begin{itemize}[itemsep=1pt,topsep=0pt, leftmargin=0mm]

\item \textbf{Input.} A subset $B\subset R^+$ \wrt a type 1 basis $\md{\e_0,\ldots,\e_r}_\Z$.

\item \textbf{Output.} True if $B$ is a root base of a root subsystem and False otherwise.

\item \textbf{Method.} 
Return True if $B$ is a linear independent set in $\md{\e_0,\ldots,\e_r}_\R$,
if the edges of $\FD(B)$ that are not self-loops are labeled $1$
and if all roots in $\md{B}_\R\cap R^+$ are a positive linear combination of elements in $B$.
\END
\end{itemize}
\end{algorithm}

\begin{algorithm}
({\tt seek\_bases})
\label{alg:seek}
\begin{itemize}[itemsep=1pt,topsep=0pt, leftmargin=0mm]

\item \textbf{Input.} An integer $1\leq r\leq 8$ and a subset $S\subset R^+$.

\item \textbf{Output.} 
A set $\Psi\subset \set{B\subset S}{ B \text{ is a root base and } |B|\leq r}$ \st
$\emptyset\in \Psi$
and for all root bases $B'\subset S$ there exists $B\in \Psi$ \st $\FC(B)\cong \FC(B')$.
Moreover, $\FC(B)\ncong \FC(B'')$ for all distinct $B,B''\in \Psi$.
The elements of the sets in $\Psi$ are represented \wrt a type 1 basis $\md{\e_0,\ldots \e_r}_\Z$.

\item \textbf{Method.} 
Let $\CT$ be the set of Dynkin diagrams 
with at most $r$ vertices and
with components of types $A_\ell$, $D_\ell$ and/or $E_\ell$
(including the empty graph $A_0$). 
Set $\Psi:=\{\}$. We do the following procedure for all $D\in \CT$ where $D$
has vertices $\{v_1,\ldots, v_n\}$.
We loop through all $b_1\in S$ for vertex $v_1$.
For each $b_1\in S$ we loop through all $b_2\in S$ such that $b_1\cdot b_2=1$
if $(v_1,v_2)$ is an edge of $D$ and $b_1\cdot b_2=0$ otherwise.
We recursively continue this procedure until finally $B=\{b_1,b_2,\ldots,b_n\}$
\st $\FD(B)\cong D$ as graphs (if such $B$ exists). We add $B$ to $\Psi$ if
$\FC(B)\neq \FC(B')$ for all $B'\in \Psi$. We return $\Psi$ after all possible $D\in \CT$ and $B\subset S$ are exhausted
with this method. 
\END
\end{itemize}
\end{algorithm}

\begin{algorithm}
({\tt get\_bas})
\label{alg:bas}
\begin{itemize}[itemsep=1pt,topsep=0pt, leftmargin=0mm]

\item \textbf{Input.} An integer $1\leq r\leq 8$. 

\item \textbf{Output.} 
A set $\Psi\subset \set{B\subset R^+}{ B \text{ is a root base and } |B|\leq r}$ \st
$\emptyset\in\Psi$ and
for all root bases $B'\subset R^+$ with $|B'|\leq r$ there exists $B\in \Psi$ \st $\FC(B)\cong \FC(B')$.
Moreover, $\FC(B)\ncong \FC(B'')$ for all distinct $B,B''\in \Psi$.
We represent all elements of sets in $\Psi$ \wrt a type 1 basis $\md{\e_0,\ldots \e_r}_\Z$.

\newpage
\item {\bf Method.} 
\begin{enumerate}[itemsep=1pt,topsep=0pt, leftmargin=0mm]
\item[]
{\bf if} $r=1$ {\bf then return} $\{\emptyset\}$

\item[] $\Psi:=\{\}$
\item[] $\Psi_1:=${\tt get\_bas}$(r-1)$
\item[] $\Psi_2:=$\texttt{seek\_bases}$(r,\set{ c\in R^+}{ c\cdot \e_r\neq 0 })$

\item[]
{\bf for} $(B_1,B_2)\in \Psi_1\times \Psi_2$ {\bf do}
\begin{enumerate}[itemsep=1pt,topsep=0pt, leftmargin=8mm]
\item[]
{\bf if} {\tt is\_root\_base}($B_1\cup B_2$) {\bf and} $\FC(B_1\cup B_2)\ncong \FC(B')$ for all $B'\in \Psi$ 
{\bf then} 
\begin{enumerate}[itemsep=1pt,topsep=0pt, leftmargin=8mm]
$\Psi:=\Psi\cup \{B_1\cup B_2\}$
\end{enumerate}
\end{enumerate}

\item[] {\bf return} $\Psi$
\END
\end{enumerate}

\end{itemize}
\end{algorithm}

\begin{algorithm}
({\tt get\_inv})
\label{alg:inv}
\begin{itemize}[itemsep=1pt,topsep=0pt, leftmargin=0mm]

\item \textbf{Input.} An integer $1\leq r\leq 8$. 

\item \textbf{Output.} 
A set $\Psi$ of pairs $(A,M)$, where $A\subset R^+$ is a root base
and $M \in \Z^{(r+1)\times(r+1)}$ is 
an involutory matrix such that the elements in $A$ form a basis for the 
eigenspace of~$M$ for eigenvalue $-1$. 
For each such pair $(A',M')$ there exists $(A,M)\in \Psi$ \st $\FC(\emptyset,M)\cong \FC(\emptyset,M')$.
Moreover, for each distinct $(A,M), (A'',M'')\in \Psi$ 
we ensured that $\FC(\emptyset,M)\ncong \FC(\emptyset,M'')$. Both $A$ and $M$ are represented \wrt 
a type 1 basis $\md{\e_0,\ldots,\e_r}_\Z$.

\item \textbf{Method.} 
We set $\Psi:=\{\}$, $I$ is the identity matrix, $\k:=-3\e_0+\e_1+\ldots+\e_r$
and $J$ is the matrix corresponding to the inner product for the type 1 basis.

\item[] {\bf for} $A\in$ {\tt get\_bas($r$)} {\bf do}
\begin{enumerate}[itemsep=1pt,topsep=0pt, leftmargin=8mm]
\item We represent the elements in $A$ as rows of a matrix.
The columns of matrix $K$ form a basis for the kernel of $A\cdot J$.
Let $Q:=(A^\top|K)$ be the matrix whose columns are the union of the columns of $A^\top$ and $K$.
\item Set
$M:=Q\cdot D\cdot Q^{-1}$,
where $D$ is the diagonal matrix with the first $|A|$ entries~$-1$ and the remaining entries~$1$.
\item \textbf{if} 
$M\in \Z^{(r+1)\times(r+1)}$ 
\textbf{and} 
$M^\top\cdot J\cdot M=J$
\textbf{and} 
$M\cdot M=I$
\textbf{and} 
$M(\k)=\k$
\textbf{and}
$\FC(\emptyset,M)\ncong\FC(\emptyset,M')$ for all $(A',M')\in\Psi$
\textbf{then} $\Psi:=\Psi\cup\{ (A, M) \}$

\end{enumerate}
\item[] {\bf return} $\Psi$
\END
\end{itemize}
\end{algorithm}

\newpage
\begin{algorithm}
({\tt get\_cls})
\label{alg:cls}
\begin{itemize}[itemsep=1pt,topsep=0pt, leftmargin=0mm]
\item {\bf Input.} 
An integer $1\leq r\leq 8$.

\item {\bf Output.}
The set
\\
$\Psi:=\{[\BN(X)]~|~X \text{ is a weak del Pezzo surface},~\text{rank}(N(X))=r+1,~\h=-\k\}$,
\\
where 
the equivalence class $[\BN(X)]$ is represented by a triple $(B,A,M)$
where $B,A\subset R^+$ represent $B(X),A(X)\subset R(X)$
\wrt a type 1 basis $\md{\e_0,\ldots,\e_r}_\Z$
and $M\in \Z^{(r+1)\times(r+1)}$ represents $\sigma_*\c N(X)\to N(X)$. 

\item {\bf Method.} Let $I$ denote the $(r+1)\times(r+1)$ identity matrix.
\begin{enumerate}[itemsep=1pt,topsep=0pt, leftmargin=0mm]
\item[] {\bf if} $r=1$ {\bf then return} $\{(\emptyset,\emptyset, I)\}$ 
\item[] $\Psi:=\{\}$
\item[] {\bf for} $(A,M)\in$ {\tt get\_inv}$(r)$ {\bf do}
\begin{enumerate}[itemsep=1pt,topsep=0pt, leftmargin=8mm]
\item[] $\Psi_1:=\set{ B }{(B,A,M')\in \text{\tt get\_cls}(r-1) \text{ for some } B \text{ and } M' }$ 
\item[] $\xi:=\set{ c\in R^+ }{ M(c)=c }$
\item[] Construct $\Phi$ \st $\Phi\,\dot\cup\,\set{M(c)}{c\in\Phi}= R^+\setminus\xi$.
\item[] {\bf if}~~$\Psi_1=\{\}$~~{\bf then}
\begin{enumerate}[itemsep=1pt,topsep=0pt, leftmargin=8mm]
\item[] $\Psi_1:=\{\emptyset\}$
\end{enumerate}
\item[] {\bf else}
\begin{enumerate}[itemsep=1pt,topsep=0pt, leftmargin=8mm]
\item[] $\xi:=\set{c\in\xi}{ c\cdot\e_r\neq 0}$
\item[] $\Phi:=\set{c\in\Phi}{ c\cdot\e_r\neq 0}$
\end{enumerate}
\item[] $\Psi_2:=$ {\tt seek\_bases}$(r,\xi)$ 
\item[] $\Psi_3:=$ {\tt seek\_bases}$(r,\Phi)$ 
\item[] {\bf for} $(B_1,B_2,B_3)\in \Psi_1\times \Psi_2\times \Psi_3$ {\bf do}
\begin{enumerate}[itemsep=1pt,topsep=0pt, leftmargin=8mm]
\item[] $B:=B_1\cup B_2\cup B_3\cup \set{M(c)}{c\in B_3}$
\item[] {\bf if} {\tt is\_root\_base}($B$) {\bf and} $\FC(B,M)\ncong \FC(B',M)$ for all $(B',A,M)\in \Psi$ 
{\bf then} $\Psi:=\Psi\cup \{(B,A,M)\}$
\end{enumerate}
\end{enumerate}
\item[] {\bf return} $\Psi$
\END
\end{enumerate}
\end{itemize}
\end{algorithm}

\begin{remark}
The Dynkin diagrams of the root bases in the
outputs of \ALG{bas} and \ALG{inv} for input $1\leq r\leq 8$,
are stated in \citep[Corollary~2.1]{wal1} and \citep[Chapter~8]{dol1}.
The Dynkin diagrams of the root bases in the output of \ALG{bas} can alternatively be obtained using the procedure of
Borel-Dynkin-Siebenthal (see \citep[Section~8.2.3 and Section~8.7.2]{dol1} for further references). 
\END
\end{remark}

\begin{example}
\label{exm:deg6}
We consider the output of \ALG{cls} with input $r=3$.
The corresponding row numbers are 3 until 14 in \TAB{nsl}.
In \TAB{crem} we depict the Cremona invariant $\FC(X)$ for each $[\BN(X)]$ in the output 
of \ALG{cls}. 

Recall from \PRP{coord} and \NTN{C} that 
\begin{Mlist}
\item $N(X)\cong\md{\e_0,\e_1,\e_2,\e_3}_\Z$ with $-\k=\h=3\e_0-\e_1-\e_2-\e_3$,
\item $E(X)=\{\e_1,~\e_2,~\e_3,~\e_0-\e_1-\e_2,~\e_0-\e_1-\e_3,~\e_0-\e_2-\e_3\}$,
\item $F(X)=\{\e_0-\e_1,~\e_0-\e_2,~\e_0-\e_3\}$,
and 
\item $R^+=\{\e_1-\e_2,~ \e_2-\e_3,~ \e_1-\e_3,~ \e_0-\e_1-\e_2-\e_3\}$.
\end{Mlist}
We know from \LEM{G} and \LEM{R} that $G(X)=F_\R(X)$ and $A(X),B(X)\subset R^+$.
Notice that $B(X)$ can be recovered from the square vertices in \TAB{crem}.

\begin{table}[!ht]
\caption{
The twelve Cremona invariants of Neron-Severi lattices of weak del Pezzo surfaces of degree 6
(self-loops of vertices are omitted).
}
\label{tab:crem}
\centering
\def\xa{-0.5}\def\ya{+1.0}
\def\xb{+0.5}\def\yb{+1.0}
\def\xc{+1.0}\def\yc{+0.0}
\def\xd{+0.5}\def\yd{-1.0}
\def\xe{-0.5}\def\ye{-1.0}
\def\xf{-1.0}\def\yf{+0.0}
\def\xg{+0.0}\def\yg{+0.0}
\def\xh{+0.0}\def\yh{+0.5}
\def\xi{+0.5}\def\yi{+0.25}
\def\A{(\xa,\ya)}
\def\B{(\xb,\yb)}
\def\C{(\xc,\yc)}
\def\D{(\xd,\yd)}
\def\E{(\xe,\ye)}
\def\F{(\xf,\yf)}
\def\G{(\xg,\yg)}
\def\H{(\xh,\yh)}
\def\I{(\xi,\yi)}

\def\rb{0.15}
\def\GG{(\xg-\rb/2,\yg-\rb/2)}
\def\HH{(\xh-\rb/2,\yh-\rb/2)}
\def\II{(\xi-\rb/2,\yi-\rb/2)}

\newcommand{\vc}[1] {\draw[draw=black, fill=red!20, line width=0.2mm] #1 circle [radius=1mm];}
\newcommand{\vb}[1] {\draw[draw=black, fill=blue!20] #1 rectangle ++(\rb,\rb);}

\newcommand{\vv}[1] {\draw[draw=gray!70, fill=gray!20, line width=0.1mm] #1 circle [radius=1mm];}
\newcommand{\vw}[1] {\draw[draw=white, fill=white, line width=0.1mm] #1 circle [radius=1mm];}

\newcommand{\eb}[2] {\draw[line width=0.2mm, draw=black] #1 -- #2;}
\newcommand{\er}[2] {\draw[line width=0.2mm, draw=red, densely dotted] #1 -- #2;}
\newcommand{\eg}[2] {\draw[line width=0.2mm, draw=black!20!green, densely dashed] #1 -- #2;}
\newcommand{\egc}[3][] {\draw[line width=0.2mm, draw=black!20!green, densely dashed] #2 to [#1] #3;}

\def\XL{-0.85}
\def\YL{0.5}

\def\P{(0,-1.5)}

\begin{tabular}{@{}c@{~~}c@{~~}c@{~~}c@{~~}c@{}}
\begin{tikzpicture}
\vv \A; \vc \B; \vv \C; \vc \D; \vv \E; \vc \F;

\node[red] at (\xb-0.40,\yb-0.25) {\footnotesize$\e_0-\e_1-\e_3$};
\node[red] at (\xd-0.40,\yd+0.25) {\footnotesize$\e_0-\e_2-\e_3$};
\node[red] at (\xf+1.00,\yf+0.00) {\footnotesize$\e_0-\e_1-\e_2$};

\node[black] at \P {vertex labels};
\end{tikzpicture}
&
\begin{tikzpicture}
\vc \A; \vv \B; \vc \C; \vv \D; \vc \E; \vv \F;
\vb \GG; \vb \HH; \vb \II;

\node[red] at (\xa+0.30,\ya-0.1) {\footnotesize$\e_1$};
\node[red] at (\xc-0.05,\yc+0.25) {\footnotesize$\e_3$};
\node[red] at (\xe+0.30,\ye+0.1) {\footnotesize$\e_2$};
\node[blue] at (\xg-0.05,\yg-0.25) {\tiny$\e_0-\e_1-\e_2-\e_3$};
\node[blue] at (\xh-0.55,\yh+0.1) {\tiny$\e_1-\e_2$};
\node[blue] at (\xi-0.55,\yi+0.0) {\tiny$\e_2-\e_3$};

\node[black] at \P {vertex labels};
\end{tikzpicture}
&
\begin{tikzpicture}
\vw \A; \vw \B; \vw \C; \vw \D; \vw \E; \vw \F;

\eb{(\XL,\YL+0.0)}{(\XL+1,\YL+0.0)}; \node[black, right]          at (\XL+1,\YL+0.0) {:\,$1$};
\er{(\XL,\YL-0.5)}{(\XL+1,\YL-0.5)}; \node[red, right]            at (\XL+1,\YL-0.5) {:\,$-1$};
\eg{(\XL,\YL-1.0)}{(\XL+1,\YL-1.0)}; \node[black!20!green, right] at (\XL+1,\YL-1.0) {:\,$\infty$};

\node[black] at \P {edge labels};
\end{tikzpicture}
&
\begin{tikzpicture}
\eb \A \B; \eb \B \C; \eb \C \D; \eb \D \E; \eb \E \F; \eb \F \A;
\vc \A; \vc \B; \vc \C; \vc \D; \vc \E; \vc \F;
\node[black] at \P {$(A_0,A_0)$};
\end{tikzpicture}
&
\begin{tikzpicture}
\eb \A \B; \eb \B \C; \eb \C \D; \eb \D \E; \eb \E \F; \eb \F \A;
\eb \B \H; \eb \E \H;
\er \A \H; \er \D \H;
\vc \A; \vc \B; \vc \C; \vc \D; \vc \E; \vc \F;
\vb \HH
\node[black] at \P {$(A_0,\underline{A_1})$};
\end{tikzpicture}
\\
\begin{tikzpicture}
\eb \A \B; \eb \B \C; \eb \C \D; \eb \D \E; \eb \E \F; \eb \F \A;
\eb \A \G; \eb \C \G; \eb \E \G;  
\er \B \G; \er \F \G; \er \D \G;
\vc \A; \vc \B; \vc \C; \vc \D; \vc \E; \vc \F;
\vb \GG;
\node[black] at \P {$(A_0,\underline{A_1})$};
\end{tikzpicture}
&
\begin{tikzpicture}
\eb \A \B; \eb \B \C; \eb \C \D; \eb \D \E; \eb \E \F; \eb \F \A;
\eb \B \H; \eb \E \H;
\eb \H \I; \eb \C \I; \eb \F \I;
\er \A \H; \er \D \H;
\er \B \I; \er \E \I;
\vc \A; \vc \B; \vc \C; \vc \D; \vc \E; \vc \F;
\vb \HH; \vb \II;
\node[black] at \P {\small$(A_0,\underline{A_2})$};
\end{tikzpicture}
&
\begin{tikzpicture}
\eb \A \B; \eb \B \C; \eb \C \D; \eb \D \E; \eb \E \F; \eb \F \A;
\eb \A \G; \eb \C \G; \eb \E \G;  
\er \B \G; \er \F \G; \er \D \G;

\eb \B \H; \eb \E \H;
\er \A \H; \er \D \H;

\vc \A; \vc \B; \vc \C; \vc \D; \vc \E; \vc \F;
\vb \GG; \vb \HH;

\node[black] at \P {$(A_0,2\underline{A_1})$};
\end{tikzpicture}
&
\begin{tikzpicture}
\eb \A \B; \eb \B \C; \eb \C \D; \eb \D \E; \eb \E \F; \eb \F \A;
\eb \A \G; \eb \C \G; \eb \E \G;  
\eb \B \H; \eb \E \H;
\eb \H \I; \eb \C \I; \eb \F \I;
\er \B \G; \er \F \G; \er \D \G;
\er \A \H; \er \D \H;
\er \B \I; \er \E \I;
\vc \A; \vc \B; \vc \C; \vc \D; \vc \E; \vc \F;
\vb \GG; \vb \HH; \vb \II;
\node[black] at \P {\small$(A_0,\underline{A_1}+\underline{A_2})$};
\end{tikzpicture}
&
\begin{tikzpicture}
\eb \A \B; \eb \B \C; \eb \C \D; \eb \D \E; \eb \E \F; \eb \F \A;
\eg \A \E; \eg \B \D;
\vc \A; \vc \B; \vc \C; \vc \D; \vc \E; \vc \F;
\node[black] at \P {$(A_1,A_0)$};
\end{tikzpicture}
\\
\begin{tikzpicture}
\eb \A \B; \eb \B \C; \eb \C \D; \eb \D \E; \eb \E \F; \eb \F \A;
\eb \A \G; \eb \C \G; \eb \E \G;  
\er \B \G; \er \F \G; \er \D \G;
\eg \A \E; \eg \B \D;
\vc \A; \vc \B; \vc \C; \vc \D; \vc \E; \vc \F;
\vb \GG;
\node[black] at \P {$(A_1,\underline{A_1})$};
\end{tikzpicture}
&
\begin{tikzpicture}
\eb \A \B; \eb \B \C; \eb \C \D; \eb \D \E; \eb \E \F; \eb \F \A;
\eg \A \D; \eg \B \E; \eg \C \F;
\vc \A; \vc \B; \vc \C; \vc \D; \vc \E; \vc \F;
\node[black] at \P {$(A_1',A_0)$};
\end{tikzpicture}
&
\begin{tikzpicture}
\eb \A \B; \eb \B \C; \eb \C \D; \eb \D \E; \eb \E \F; \eb \F \A;
\eb \B \H; \eb \E \H;
\er \A \H; \er \D \H;
\eg \A \D; \eg \B \E; \eg \C \F;
\vc \A; \vc \B; \vc \C; \vc \D; \vc \E; \vc \F;
\vb \HH;
\node[black] at \P {$(A_1',\underline{A_1})$};
\end{tikzpicture}
&
\begin{tikzpicture}
\eb \A \B; \eb \B \C; \eb \C \D; \eb \D \E; \eb \E \F; \eb \F \A;
\eb \B \H; \eb \E \H;
\eb \H \I; \eb \C \I; \eb \F \I;
\er \A \H; \er \D \H;
\er \B \I; \er \E \I;
\eg \A \D; \eg \B \E; \eg \C \F;
\vc \A; \vc \B; \vc \C; \vc \D; \vc \E; \vc \F;
\vb \HH; \vb \II;
\node[black] at \P {$(A_1',\underline{A_2})$};
\end{tikzpicture}
&
\begin{tikzpicture}
\eb \A \B; \eb \B \C; \eb \C \D; \eb \D \E; \eb \E \F; \eb \F \A;
\egc[out=-45, in=180+45]{\A}{(\xb,\yb)}; \egc[out= 45, in=180-45]{\E}{(\xd,\yd)};
\eg \F \C; 
\vc \A; \vc \B; \vc \C; \vc \D; \vc \E; \vc \F;
\node[black] at \P {$(2A_1,A_0)$};
\end{tikzpicture}
\end{tabular}
\end{table}

The output of \ALG{bas} with $r=3$ consists of 
$\emptyset$, 
$\{\e_1-\e_2\}$,
$\{\e_0-\e_1-\e_2-\e_3\}$, 
$\{\e_1-\e_2, \e_2-\e_3\}$,
$\{\e_1-\e_2, \e_0-\e_1-\e_2-\e_3\}$ and
$\{\e_1-\e_2, \e_2-\e_3, \e_0-\e_1-\e_2-\e_3\}$.

The output of \ALG{inv} with $r=3$ is a set consisting of the following pairs
(here the matrix $M$ is represented as
$[\sigma_*(\e_0),\sigma_*(\e_1),\sigma_*(\e_2),\sigma_*(\e_3)]$):
\begin{Mlist}
\item $(\emptyset,[\e_0,\e_1,\e_2,\e_3])$,
\item $(\{\e_1-\e_2\},[\e_0,\e_2,\e_1,\e_3])$,
\item $(\{\e_0-\e_1-\e_2-\e_3\},[2\e_0-\e_1-\e_2-\e_3,\e_0-\e_2-\e_3,\e_0-\e_1-\e_3,\e_0-\e_1-\e_2])$,
\item $(\{\e_1-\e_2,\e_0-\e_1-\e_2-\e_3\},[2\e_0-\e_1-\e_2-\e_3,\e_0-\e_1-\e_3,\e_0-\e_2-\e_3,\e_0-\e_1-\e_2])$, 
\end{Mlist}
where $\FD(A(X))$ 
is $A_0$, $A_1$, $A_1'$ and $2A_1$, \resp.
Notice that $\FD(A(X))$ does not uniquely determine $\sigma_*$ up to conjugacy,
so we added a $'$-symbol in case $\sigma_*(\e_0)\neq \e_0$.

If $A=\{\e_1-\e_2,\e_2-\e_3,\e_0-\e_1-\e_2-\e_3\}$ at step (1) in \ALG{inv}, then 
the kernel of $A\cdot J$ has basis $\{\e_0-\e_1-\e_2-\e_3,~\k\}$
so that at step (2):
\begin{center}
\tiny
$
Q=
\begin{bmatrix}
 0 &  0 &  1 & -3 \\
 1 &  0 & -1 & 1  \\
-1 &  1 & -1 & 1  \\
 0 & -1 & -1 & 1  \\
\end{bmatrix}
$
,\qquad
$
D=
\begin{bmatrix}
-1 &  0 &  0 & 0 \\
 0 & -1 &  0 & 0 \\
 0 &  0 & -1 & 0 \\
 0 &  0 &  0 & 1 \\
\end{bmatrix}
$
,\qquad
$
M=Q\cdot D\cdot Q^{-1}=
\begin{bmatrix}
 2 &    1 &    1 &    1 \\ 
-1 & -\frac{4}{3} & -\frac{1}{3} & -\frac{1}{3} \\
-1 & -\frac{1}{3} & -\frac{4}{3} & -\frac{1}{3} \\ 
-1 & -\frac{1}{3} & -\frac{1}{3} & -\frac{4}{3} 
\end{bmatrix}
$.
\end{center}
Thus $M$ does in this case not define a unimodular involution
and is not added to the output $\Psi$ at step (3). 
When we follow \ALG{cls} with the output of \ALG{inv} we
recover \TAB{crem} and rows 3 until 14 in \TAB{nsl}.
For example, 
suppose that 
$(A,M)=(\{\e_1-\e_2\},[\e_0,\e_2,\e_1,\e_3])$
in the outer for-loop of \ALG{cls}.
We find that {\tt get\_cls}$(2)=\{(\emptyset,\emptyset,I), (\e_1-\e_2,\emptyset,I),(\emptyset,\e_1-\e_2,[\e_0,\e_2,\e_1])\}$
so that $\Psi_1=\{\emptyset\}\neq\emptyset$.
We have $\xi=\{\e_0-\e_1-\e_2-\e_3\}$ and $\Phi=\{\e_1-\e_3\}$ where $M(\e_1-\e_3)=\e_2-\e_3$.
Thus $\Psi_2=\{\emptyset,\e_0-\e_1-\e_2-\e_3\}$ and $\Psi_3=\{\emptyset,\e_1-\e_3\}$.
If $(B_1,B_2,B_3)=(\emptyset,\emptyset,\e_1-\e_3)$, then $B$ is not a root base, since $(\e_1-\e_3)\cdot(\e_2-\e_3)\notin\{0,1\}$.
\END
\end{example}

\begin{lemma}
\label{lem:root}
The output specification of \ALG{cls} is correct.
\end{lemma}

\begin{proof}
The correctness of both \ALG{basis} and \ALG{seek} is a direct consequence of the definitions.
It follows from \LEM{R} that we only have to consider roots in~$R^+$.

Notice that the columns of $Q$ at step 1 of \ALG{inv} 
are eigenvectors of $M$ so that $M\cdot Q=D\cdot Q$.
The correctness of \ALG{inv} is thus a straightforward consequence of 
\PRP{sub}d and \RMK{A}.

For \ALG{bas} we recall from \LEM{blowdown} that if $Y\to X$ is the smooth model
of a weak del Pezzo surface $X$, 
then $Y$ is the complex blowup of a weak del Pezzo surface~$X'$. 
Thus by \PRP{coord}a,
$N(X)\cong N(X')\oplus \md{\e_r}_\Z$, $N(X')\cong \md{\e_0,\ldots,\e_{r-1}}_\Z$
and $B(X)$ is the union of $B(X')$ (represented by $B_1$) 
with $B_2\subset \set{c\in R(X)}{c\cdot \e_r\neq 0}$.
We know from \PRP{sub}b that $B(X)$ and $B(X')$ are root bases. 
Thus $B_2$ is a root base as well and the correctness of \ALG{bas} follows.

For \ALG{cls} we apply the same argument as we used for \ALG{bas}.
However, in this case $B(X')\subset B(X)$ only if $\sigma_*(\e_r)=\e_r$, 
which means that $\Psi_1\neq \emptyset$.
Notice that $A(X)$ determines $\sigma_*$ only up to a choice of a basis for
the sublattice $\set{c\in N(X)}{ c\cdot a=0 \text{ for all } a\in A(X) }$.
It follows from \LEM{R} that $B(X)\subset R^+$ for any such choice of $\sigma_*$.
However, we need to exhaust all possible root bases $B(X)\subset R^+$ \st $\sigma_*(B(X))=B(X)$.
Any such candidate $B$ for $B(X)$ can be constructed from 
some $(B_1,B_2,B_3)\in \Psi_1\times \Psi_2\times \Psi_3$.
We know from \LEM{h0} that $B(X)$ uniquely determines $h^0\c N(X)\to\Z_{\geq0}$.
We conclude from \PRP{cremona} that 
the output specification of \ALG{cls} is correct.
\end{proof}

\begin{proof}[Proof of \THM{nsl}]
Notice that 
\LEM{G}[a,b,c,d,e,f]
correspond to the rows (i), (ii), (iii), (iv), (v) and (2) in \TAB{nsl}, \resp.
By \PRP{coord}c, we can compute $E_\R(X)$ and $F_\R(X)$ from $B(X)$.
Thus the proof is a direct consequence of \LEM{G} and \LEM{root}.
It follows from \citep[Section~8.7.1 and Theorem~8.8.1]{dol1} that if $X$ has 
canonical degree $\leq 2$, then $\BN(X)$ is not characterized by
row 176, 434, 453, 455 or 505. 
\end{proof}

We explain in the following example how to read \TAB{nsl}.

\begin{example}
\label{exm:howto}
Suppose that $X\subset\P^3$ is a ring cyclide as in \FIG{simple}
so that $X$ is a quartic weak del Pezzo surface that is 
covered by four simple families of conics
and $X$ contains neither real lines nor real isolated singularities.
It follows from \THM{nsl} that $\BN(X)$ is characterized by a row in \TAB{nsl}. 
The table headers 
$d$, $D(A)$, $D(B)$, $\#E$ and $\#G$
correspond to
$\deg X$, $\FD(A(X))$, $\FD(B(X))$, $|E_\R(X)|$ and $|G(X)|$, \resp. 
The components of $\FD(B(X))$ that are preserved by $\sigma_*$ are 
underlined in table.
It follows from \PRP{class} that $d=4$, the $D(B)$ entry has no underlined components, $\#E=0$ and $\#G=4$.
Hence, $\BN(X)$ must be defined by row number 67 \st $\FD(A(X))=2A_1'$ and $\FD(B(X))=4A_1$.
It follows from \PRP{coord}a that $N(X)\cong\md{\e_0,\ldots,\e_5}_\Z$ and $-\k=\h=3\e_0-\e_1-\ldots-\e_5$.
We recover an explicit description for
$\sigma_*\c N(X)\to N(X)$ and $B(X)$ from 
the entry 
$\frac{9}{5} \frac{3}{6} \frac{3}{5} \frac{3}{4} \frac{0}{5} \frac{0}{4} || \frac{1}{1} \frac{1}{6} \frac{3}{9} \frac{4}{2} $, 
which can be decoded using the dictionary at \TAB{nsl}. 
We find that
$\sigma_*(\e_0)=2\e_0-\e_1-\e_2-\e_3$,     
$\sigma_*(\e_1)=\e_0-\e_2-\e_3$,     
$\sigma_*(\e_2)=\e_0-\e_1-\e_3$,     
$\sigma_*(\e_3)=\e_0-\e_1-\e_2$,    
$\sigma_*(\e_4)=\e_5$
and $B(X)=\{\e_3-\e_4,                 
\e_2-\e_5,              
\e_0-\e_1-\e_3-\e_4,     
\e_0-\e_1-\e_2-\e_5\}$.
Recall from \LEM{h0} that $B(X)$ uniquely recovers
$h^0\c N(X)\to \Z_{\geq 0}$ and thus we determined $\BN(X)$ up to isomorphism of Neron-Severi lattices.
We have $E_\star(X)=\{\e_1,\e_0-\e_2-\e_3, \e_4,\e_5\}$ 
and $F_\R(X)=\{\e_0-\e_1,\e_0-\e_2,\e_0-\e_3,2\e_0-\e_2-\e_3-\e_4-\e_5\}$,
by \PRP{coord}, and we know that $G(X)=F_\R(X)$ by \LEM{G}.
Recall from \RMK{G} that the simple family graph $\CG(X)$ 
in \FIG{simple} is uniquely determined by 
the incidence diagram $\FD(G(X))$.
Notice that $\e_0-\e_1$ and $2\e_0-\e_2-\e_3-\e_4-\e_5$ are the classes of Villarceau
circles.
\END
\end{example}

\begin{remark}
\label{rmk:AA}
Suppose that $X$ and $X'$ are anticanonical models of smooth weak del Pezzo surfaces.
It is observed in \citep[Section~2]{wal1} that $\BN(X)\cong \BN(X')$ 
does not imply that $\FD(A(X))\cong \FD(A(X'))$.
Indeed, using the Cremona invariant we find that if $\deg X=\deg X'$
and both $\FD(A(X))$ and $\FD(A(X'))$ are listed below at the same bullet, then $\BN(X)\cong \BN(X')$:
\begin{Mlist}
\item $4A_1$, $D_4$.  
\item $5A_1$, $A_1+D_4$.
\item $6A_1$, $2A_3$, $2A_1+D_4$, $D_6$.  
\item $7A_1$, $A_7$, $A_2+A_5$, $A_1+D_6$, $3A_1+D_4$, $A_1+2A_3$, $E_7$.
\item $8A_1$, $A_8$,  $4A_2$, $2A_4$, $2D_4$, $A_3+D_5$, $A_1+A_7$, $A_2+E_6$, $A_1+E_7$, $4A_1+D_4$, $2A_1+D_6$, $2A_1+2A_3$, $A_1+A_2+A_5$, $D_8$, $E_8$. \END
\end{Mlist}
\end{remark}

\section{Simple family graphs via adjunction}
\label{sec:simple}

In this section we prove \THM{coarse} and \THM{pmz}.
We start by recalling a version of adjunction 
as defined in \citep[Section~2.4]{nls-algo-min-fam}.

A \df{ruled pair} is defined as a pair $(Y,\h)$
where $Y$ is a smooth birationally ruled surface, 
$\h$ is nef and there are no classes of complex $(-1)$-curves that are orthogonal to~$\h$.

Suppose that $Y_0\to X$ is the smooth model
of a surface $X\subset \P^n$ that is covered by a simple family.
Let $h_0\in N(X)$ be the class of hyperplane sections.
In this case, $(Y_0,\h_0)$ is a ruled pair.
If $h^0(\h_0+\k_0)>1$, then we consider the birational morphism
$\mu\c Y_0\to Y_1$ that contracts all complex $(-1)$-curves $E\subset Y_0$
\st $(\h_0+\k_0)\cdot [E]=0$. An \df{adjoint relation} is defined as
\[
\arrow{(Y_0,\h_0)}{\mu}{(Y_1,\h_1):=(\mu(Y_0),\mu_*(\h_0+\k_0))}.
\]
This relation $\mu$ is unique up to biregular isomorphism and $(Y_1,\h_1)$ is again a ruled pair \citep[Proposition~1]{nls-algo-min-fam}.
In particular, the adjoint relation satisfies the following property:

\begin{lemma}
\label{lem:dis}
If $\arrow{(Y,\h)}{\mu}{(Y',\h')}$ is an adjoint relation \st $\h'^2>0$
and $E,E'\subset Y$ are complex $(-1)$-curves that are contracted by $\mu\c Y\to Y'$,
then $[E]\cdot [E']=0$.
\end{lemma}

\begin{proof}
We have that $\mu^*\h'=\h+\k$, since the classes of complex $(-1)$-curves that are contracted by $\mu$
are by definition orthogonal to $\h+\k$. Thus $(\h+\k)^2=(\mu^*\h')^2=\h'^2>0$.
Since $(\h+\k)\cdot([E]+[E'])=0$ it follows from (HI) that $([E]+[E'])^2<0$ and thus $[E]\cdot [E']=0$.
\end{proof}

The \df{adjoint chain} for $X$ is a chain of adjoint relations 
\[
\arrow{(Y_0,\h_0)}{\mu_0}{(Y_1,\h_1)}\arrow{}{\mu_1}{}\ldots\arrow{}{\mu_{\ell-1}}{(Y_\ell,\h_\ell)},
\]
\st $h^0(\h_i+\k_i)>1$ for $0\leq i< \ell$ and $h^0(\h_\ell+\k_\ell)\leq 1$.
Although $Y_i$ for $0\leq i\leq \ell$ is considered as an abstract surface it comes equiped 
with the class of hyperplane sections $\h_i$ and the canonical class $\k_i$.
This motivates us to consider the following definition for the classes of
simple families of $Y_i$:
\[
G(Y_i):=\set{c\in N(Y_i)}{ p_a(c)=0,~  h^0(c)\geq 2 \text{ and } \h_i\cdot c=\nu_i }, 
\]
where  
$\nu_i:=\min\set{ \h_i\cdot c }{ c\in N(Y_i),~ p_a(c)=0 \text{ and } h^0(c)\geq 2}$.
Notice that $G(X)=G(Y_0)$ by definition.
Recall from \RMK{G} that we recover the simple family graph $\CG(X)$ uniquely from $G(X)$.

The following \LEM{bas} and \LEM{pp} collect results from \cite{nls-algo-min-fam}.
We use the following notation for $t\in\Z$:
$[t]:=1$ and $[t]:=0$ if $t\geq 0$ and $t<0$, \resp.

\begin{lemma}
\label{lem:bas}
If $\arrow{(Y_0,\h_0)}{\mu_0}{(Y_1,\h_1)}\arrow{}{\mu_1}{}\ldots\arrow{}{\mu_{\ell-1}}{(Y_\ell,\h_\ell)}$
is an adjoint chain, then either 
$G(Y_0)=\{f\}$ \st $\k_0\cdot f=-2$,
or
$Y_\ell$ is a weak del Pezzo surface and one of the following four cases holds:
\begin{MenumA}

\item $G(Y_0)=\set{(\mu_{\ell-1}\circ\ldots\circ \mu_0)^*(g)}{g\in G(Y_\ell),~ \k_\ell\cdot g=-2}\neq\emptyset$.

\item $N(Y_i)$ admits a real type~1 basis for all $0\leq i\leq \ell$ \st
\begin{align*}
\h_i &= (\alpha_0-3i)\e_0-[\alpha_1-i](\alpha_1-i)\e_1-\ldots-[\alpha_r-i](\alpha_r-i)\e_r,
\\
\k_i &= -3\e_0+[\alpha_1-i]\e_1+\ldots+[\alpha_r-i]\e_r, 
\end{align*}
where $\alpha_0-3i>0$ and $\alpha_1,\ldots,\alpha_r\in\Z_{>0}$. Moreover, $c\in G(Y_i)$ 
is up to permutation of $(\e_j)_{j>0}$
an element of 
$\{\e_0,~\e_0-\e_1,~2\e_0-\e_1-\e_2-\e_3-\e_4\}$.

\item $N(Y_i)$ admits a real type~2 basis for all $0\leq i\leq \ell$ \st
\begin{align*}
\h_i &= (\alpha_0-2i)(\l_0+\l_1)-[\alpha_1-i](\alpha_1-i)\p_1-\ldots-[\alpha_r-i](\alpha_r-i)\p_r,
\\
\k_i &= -2(\l_0+\l_1)+[\alpha_1-i]\p_1+\ldots+[\alpha_r-i]\p_r, 
\end{align*}
where $\alpha_0-2i>0$ and $\alpha_1,\ldots,\alpha_r\in\Z_{>0}$.
Moreover, $c\in G(Y_i)$ 
is up to permutation of $(\p_j)_{j>0}$
an element of 
$\{\l_0+\l_1,~\l_0+\l_1-\p_1,~ \l_0+\l_1-\p_1-\p_2\}$.

\item $G(Y_0)=G(Y_\ell)=\emptyset$ and $Y_\ell$ is not $\R$-rational.

\end{MenumA}
\end{lemma}

\begin{proof}
See \citep[Lemma~1 and Lemma~2]{nls-algo-min-fam} for the case distinction
(this result is an adaptation of the classification result by Comessatti \citep[Theorem~4.6]{sil1}).
For case (a) see also \citep[Proposition~3 or Theorem~1i]{nls-algo-min-fam}. 
For the candidate classes at (b) and (c) see \citep[Theorem~1ii and Theorem~1iii]{nls-algo-min-fam}.
\end{proof}

\begin{lemma}
\label{lem:pp}
Let $\arrow{(Y_0,\h_0)}{\mu_0}{(Y_1,\h_1)}\arrow{}{\mu_1}{}\ldots\arrow{}{\mu_{\ell-1}}{(Y_\ell,\h_\ell)}$
be an adjoint chain. 
\begin{Mlist}

\mclaim{a}
If $Y_i\cong\P^2$ for some $0<i\leq \ell$ and $\mu_{i-1}$ contracts at least one real $(-1)$-curve,
then $G(Y_{i-1})=\set{ \e_0-\e_j  }{ \mu_*(\e_j)=\e_j \text{ for } 1\leq j\leq r }$.

\mclaim{b}
If $G(Y_i)=\set{g\in G(Y_i)}{ \k_i\cdot g=-2}$ for some $0<i\leq \ell$, 
then $G(Y_{i-1})=\set{\mu_{i-1}^*(g)}{g\in G(Y_i) \text{ and } \k_i\cdot g=-2}$.

\mclaim{c} $h^0(g)=-\k_0\cdot g-1$ for all $g\in G(Y_0)$.
\end{Mlist}
\end{lemma}

\begin{proof}
For \refmclaim{a}, \refmclaim{b} 
and \refmclaim{c}, see 
\citep[Lemma~12]{nls-algo-min-fam},
\citep[Proposition~3]{nls-algo-min-fam} and
\citep[Corollary~1b]{nls-algo-min-fam},
\resp.
\end{proof}

\begin{lemma}
\label{lem:fix}
Let $(Y,\h)$ be a ruled pair.
\begin{Mlist}
\mclaim{a}
Suppose that $N(Y)$ admits a real type 1 basis
\st in addition $\sigma_*(\e_1)=\e_2$ and $\sigma_*(\e_3)=\e_4$.
If $\h\cdot g < \h\cdot \e_0$ where $g:=2\e_0-\e_1-\e_2-\e_3-\e_4$, 
then $g$ is the class of a 
complete rational family.

\mclaim{b}
Suppose that $N(Y)$ admits a real type 2 basis
\st in addition $\sigma_*(\l_0)=\l_1$ and $\sigma_*\{\p_1,\p_2\}=\{\p_1,\p_2\}$.
In this case $f:=\l_0+\l_1-\p_1-\p_2$ is the class of a 
complete rational family.
\end{Mlist}
\end{lemma}

\newpage
\begin{proof}
\refmclaim{a} 
We consider $g$ as the class of 
the linear series $|g|$ of conics in $\P^2$ that pass through four base points.
Suppose by contradiction that $|g|$ has a fixed component. 
If $h^0(\e_0-\e_1-\e_2-\e_3)=1$, 
then $h^0(\e_0-\e_1-\e_2-\e_3-\e_4)=1$ as well, since $\sigma_*(\e_3)=\e_4$ by assumption.
It follows that $\e_0$ is the class of the moving component 
and $\e_0-\e_1-\e_2-\e_3-\e_4$ is the class of the fixed component 
so that $h^0(\e_0-\e_1-\e_2-\e_3-\e_4)=1$. 
Since $\h\cdot g < \h\cdot \e_0$ we find that $\h\cdot(\e_0-\e_1-\e_2-\e_3-\e_4)<0$.
We arrived at a contradication as $\h$ is nef.
Therefore $|g|$ has no fixed components and thus $g$ is the class of a 
complete rational family.

\refmclaim{b}
We consider $f$ as the class of 
the linear series $|f|$ of bidegree (1,1) forms on $\P^1\times\P^1$
that pass through two base points.
If $|f|$ would have a fixed component, then 
\Wlog $\l_0$ is the class of the moving component and 
$\l_1-\p_1-\p_2$ is the class of the fixed component.
However, this implies that $\sigma_*(\l_0)=\l_0$ and thus 
the linear series of $f$ has no fixed components.
We conclude from (AF) that $f$ is as asserted.
\end{proof}

\begin{notation}
\label{ntn:surf}
We consider the following types of real abstract surfaces:
\begin{Mlist}
\item[] $\BP_0$: $\P^2$ together with its unique real structure.
\item[] $\BP_1$: $\P^2$ blownup in at least 1 real point.
\item[] $\BP_2$: $\P^2$ blownup in exactly 2 complex conjugate points.
\item[] $\BP_4$: $\P^2$ blownup in at least 4 complex conjugate points and no real points.
\item[] $\BS_0$: $\P^1\times\P^1$ together with the real structure $\sigma$ that flips the two factors.
\item[] $\BS_1$: $\BS_0$ blown up in exactly 1 real point.
\item[] $\BS_2$: $\BS_0$ blown up in at least 2 complex points.
\end{Mlist}
For example, we write $Y_0\in \BP_1$ if $Y_0$ is the blowup of $\P^2$ in at least 1 real point;  
we use the same notation for the other cases.
\END
\end{notation}

\begin{table}[!ht] 
\setstretch{1.2}
\caption{See \LEM{case} and \NTN{surf}.}
\label{tab:case}
\noindent\rule{\textwidth}{0.4pt}
Notation:
\begin{Mlist}
\item[] $\alpha:=\max(\set{0\leq i<\ell}{ Y_i\ncong Y_\ell}\cup\{-1\})$; if $\alpha=-1$, then $\alpha:=\ell$.
\item[] $\beta:=\max(\set{0\leq i\leq \ell}{ G(Y_i)=\set{g\in G(Y_i)}{ \k_i\cdot g=-2 }}\cup\{0\})$.   
\item[] $\tau:=\rnk N(Y_\beta)-1$.
\item[] $\xi_1:=\set{\e_0-\e_a}{\sigma_*\e_a=\e_a,~0< a\leq \tau}$.
\item[] $\xi_2:=\set{2\e_0-\e_1-\e_2-\e_a-\e_b}{ \sigma_*\e_a=\e_b,~ 3\leq a<b\leq \tau }$.
\item[] $\xi_3:=\set{2\e_0-\e_a-\e_b-\e_c-\e_d}{ \sigma_*\e_a=\e_b,~ \sigma_*\e_c=\e_d,~0<a<b<c<d\leq\tau}$.
\item[] $\xi_4:=\set{\l_0+\l_1-\p_1-\p_a}{ \sigma_*\p_a=\p_a,~ 0<a\leq \tau}$.
\item[] $\xi_5:=\set{\l_0+\l_1-\p_a-\p_b}{ \sigma_*\{\p_a,\p_b\}=\{\p_a,\p_b\},~ 0<a<b\leq \tau}$.
\item[] $\xi_2'\subseteq \xi_2$\quad{and}\quad$\xi_3'\subseteq \xi_3$.
\end{Mlist}
Cases:
\begin{MenumA}


\item 
$Y_\alpha\in \BP_1$, $Y_\ell\in\BP_0$ and $G(Y_0)=\xi_1$.

\item 
$Y_0\in \BP_2$ or $Y_0\in\BP_0$, and $G(Y_0)=\{\e_0\}$.

\item 
$Y_0\notin \BP_2$, $Y_\alpha\in \BP_2$  and $G(Y_0) \in \{\xi_1,~\xi_2,~\xi_2'\cup\{\e_0\},~\xi_2\cup\xi_1\}$.

\item 
$Y_\alpha\in\BP_4$, $Y_\ell\in \BP_0$ and $G(Y_0)=\{\xi_3,~\xi_3'\cup\{\e_0\},~\xi_3\cup\xi_1\}$.

\item $Y_0\in\BS_0$ and $G(Y_0)=\{\l_0+\l_1\}$.
\item $Y_0\in\BS_1$ and $G(Y_0)=\{\l_0+\l_1-\p_1\}$.
\item $Y_0\in\BS_2$, $Y_\alpha\in\BS_1$ and 
$G(Y_0)\in\{\xi_4,~\xi_5\cup\{\l_0+\l_1-\p_1\},~\{\l_0+\l_1-\p_1\}\}$.

\item $Y_\alpha\in\BS_2$, $Y_\ell\in\BS_0$ and $G(Y_0)=\xi_5$.

\end{MenumA}
\noindent\rule{\textwidth}{0.4pt}
\end{table}

\begin{lemma}
\label{lem:case}
Let $\arrow{(Y_0,\h_0)}{\mu_0}{(Y_1,\h_1)}\arrow{}{\mu_1}{}\ldots\arrow{}{\mu_{\ell-1}}{(Y_\ell,\h_\ell)}$
be an adjoint chain.
If $Y_\ell$ is a weak del Pezzo surface, 
$\set{g\in G(Y_\ell)}{ \k_\ell\cdot g=-2 }=\emptyset$ and $G(Y_\ell)\neq \emptyset$,
then $N(Y_0)$ is generated by either a real type 1 basis or a real type 2 basis
and one of the 8 cases in \TAB{case}[a-h] holds.
\end{lemma}

Before we prove \LEM{case} and \THM{coarse} let us 
first consider some explicit examples.

\newpage
\begin{example}[\THM{coarse}f]
\label{exm:l}
Suppose that 
$\arrow{(Y_0,\h_0)}{\mu_0}{(Y_1,\h_1)}$
is the adjoint chain of a surface $X\subset\P^n$.
The involution $\sigma_*\c N(Y_0)\to N(Y_0)$ is defined as
$\sigma_*(\e_0)=\e_0$ and $\sigma_*(\e_i)=\e_{i+1}$ for $i\in\{1,3,5,7\}$; and
\begin{Mlist}
\item $Y_1\in\BP_0$, $\h_1=\e_0$, $\k_1=-3\e_0$ and $G(Y_1)=\{\e_0\}$.
\item $Y_0\in\BP_4$, $\h_0=4\e_0-\e_1-\ldots-\e_8$, $\k_0=-3\e_0+\e_1+\ldots+\e_8$
and
$G(Y_0)=\{\e_0$,
$2\e_0-\e_1-\e_2-\e_3-\e_4$, 
$2\e_0-\e_1-\e_2-\e_5-\e_6$,
$2\e_0-\e_1-\e_2-\e_7-\e_8$,
$2\e_0-\e_3-\e_4-\e_5-\e_6$,
$2\e_0-\e_3-\e_4-\e_7-\e_8$,
$2\e_0-\e_5-\e_6-\e_7-\e_8\}$.
\end{Mlist}
The adjoint chain is characterized by \TAB{case}d
and the simple family graph $\CG(X)$ is characterized by
\THM{coarse}f. 
\END
\end{example}

\newpage
\begin{example}[\THM{coarse}h]
\label{exm:h}
Suppose that 
$\arrow{(Y_0,\h_0)}{\mu_0}{(Y_1,\h_1)}\arrow{}{\mu_1}{(Y_2,\h_2)}$
is the adjoint chain of a surface $X\subset\P^n$.
The involution $\sigma_*\c N(Y_0)\to N(Y_0)$ is defined as
$\sigma_*(\e_0)=\e_0$, $\sigma_*(\e_i)=\e_{i+1}$
for $i\in\{1,3,5\}$, $\sigma_*(\e_7)=\e_7$, $\sigma_*(\e_8)=\e_8$ and $\sigma_*(\e_9)=\e_9$; and
\begin{Mlist}
\item  
$Y_2\in\BP_2$, $\h_2=3\e_0-\e_1-\e_2$, $\k_2=-3\e_0+\e_1+\e_2$ 
and 
$G(Y_2)=\{\e_0\}$,

\item 
$Y_1\in\BP_4$,
$\h_1=6\e_0-2\e_1-2\e_2-\e_3-\e_4-\e_5-\e_6$,
$\k_1=-3\e_0+\e_1+\ldots+\e_6$
and 
$G(Y_1)=\{\e_0$, $2\e_0-\e_1-\e_2-\e_3-\e_4$, $2\e_0-\e_1-\e_2-\e_5-\e_6\}$,

\item 
$Y_0\notin\BP_2$,
$\h_0=9\e_0-3\e_1-3\e_2-2\e_3-2\e_4-2\e_5-2\e_6-\e_7-\e_8-\e_9$, 
$\k_0=-3\e_0+\e_1+\ldots+\e_9$,
and
$G(Y_0)=\{2\e_0-\e_1-\e_2-\e_3-\e_4$,  
$2\e_0-\e_1-\e_2-\e_5-\e_6$,
$\e_0-\e_7$,
$\e_0-\e_8$,
$\e_0-\e_9\}$.
\end{Mlist}
The adjoint chain is characterized by \TAB{case}c
and $\CG(X)$ is characterized by \THM{coarse}h.
\END
\end{example}

\begin{example}[\THM{coarse}j]
\label{exm:j}
Suppose that 
$\arrow{(Y_0,\h_0)}{\mu_0}{(Y_1,\h_1)}$
is the adjoint chain of a surface $X\subset\P^n$.
The involution $\sigma_*\c N(Y_0)\to N(Y_0)$ is defined as
$\sigma_*(\l_0)=\l_1$, 
$\sigma_*(\p_i)=\p_{i}$
for $i\in\{1,2,3\}$ and $\sigma_*(\p_i)=\p_{i+1}$ for $i\in\{4,6,8\}$; and
\begin{Mlist}
\item $Y_1\in \BS_0$, $\h_1=\l_0+\l_1$, $\k_1=-2(\l_0+\l_1)$ and
$G(X_1)=\{\l_0+\l_1\}$, 

\item 
$Y_0\in \BS_2$,
$\h_0=3(\l_0+\l_1)-\p_1-\ldots-\p_5$,
$\k_0=-2(\l_0+\l_1)+\p_1+\ldots+\p_5$ and
$G(X_0)=\{
 \l_0+\l_1-\p_1-\p_2$,
$\l_0+\l_1-\p_1-\p_3$,
$\l_0+\l_1-\p_2-\p_3$,
$\l_0+\l_1-\p_4-\p_5$,
$\l_0+\l_1-\p_6-\p_7$,
$\l_0+\l_1-\p_8-\p_9\}$.
\end{Mlist}
The adjoint chain is characterized by \TAB{case}h
and $\CG(X)$ is characterized by \THM{coarse}j. 
\END
\end{example}

\begin{proof}[Proof of \LEM{case}.]
Our strategy is to backtrace through the adjoint chain and  
determine all possible scenarios for $G(Y_i)$ for $i=\ell-1,\ell-2,\ldots,0$ subsequently.
Recall from \LEM{bas} that 
$N(Y_i)$ admits a real type 1 or 2 basis and that
the possible classes in $G(Y_i)$ 
are up to permutation of the $(\e_j)_{j>0}$ and $(\p_j)_{j>0}$
in the set
\[
\{\e_0, \e_0-\e_1, 2\e_0-\e_1-\e_2-\e_3-\e_4\}\cup
\{\l_0+\l_1, \l_0+\l_1-\p_1,  \l_0+\l_1-\p_1-\p_2\}.
\]
Moreover, we know $\k_i$ and $\h_i$ for $i>0$
if we know $\h_\ell$ and the number of complex $(-1)$-curves
that are contracted during each adjoint relation.
Let $\Upsilon_i:=\set{g\in G(Y_i)}{ \k_i\cdot g=-2}$
and recall that $\Upsilon_\ell=\emptyset$ by assumption.
\newpage
We use the following observations:
\begin{Mlist}

\item If $Y_0\cong Y_\ell$, then  
it follows from \LEM{G}[a,f,c] that
$G(Y_\ell)\in\{\{\e_0\},\{\l_0+\l_1\}\}$ and
one of the cases in \TAB{case}[b,e] holds.

\item 
If $Y_i\ncong\P^2$ and $Y_{i+1}\cong\P^2$, 
then either \TAB{case}a holds or $Y_i\in \BP_2\cup \BP_4$, by \LEM{pp}a. 
Moreover, $\h_\ell=m\,\e_0$ for some $m\in\Z_{>0}$.

\item
If $\Upsilon_i\neq\emptyset$,
then $G(Y_i)\in\{ \Upsilon_i,~\Upsilon_i\cup\{\e_0\},~\Upsilon_i\cup\{\l_0+\l_1-\p_1\} \}$
and if $G(Y_i)=\Upsilon_i$, then 
we may assume \Wlog that $i=0$, by \LEM{pp}b.

\item
If $G(Y_i)\cap\xi_1\neq\emptyset$, then $\e_0\notin G(Y_i)$
and thus $G(Y_i)\in\{\xi_1,\xi_2\cup\xi_1,\xi_3\cup\xi_1\}$ by \LEM{fix}a
and one of \TAB{case}[a,c,d] holds. 

\item
If $G(Y_i)\cap\xi_3\neq\emptyset$ and $\e_0\notin G(Y_i)$,
then $G(Y_i)\in\{\xi_2,\xi_3\}$ by \LEM{fix}a
and one of \TAB{case}[c,d] holds. 

\item
If $G(Y_i)\cap\xi_3\neq\emptyset$, $\e_0\in G(Y_i)$ and $Y_0\cong Y_i$, 
then $G(Y_0)\in\{\xi_2'\cup\{\e_0\}, \xi_3'\cup\{\e_0\} \}$
and one of \TAB{case}[c,d] holds.  We cannot apply \LEM{fix}a in this case.

\item 
If $G(Y_i)\cap\xi_3\neq\emptyset$, $\e_0\in G(Y_i)$ and $Y_0\ncong Y_i$, 
then $G(Y_0)\in\{\xi_2, \xi_3, \xi_2\cup\xi_1, \xi_3\cup\xi_1\}$ by \LEM{fix}a
and one of \TAB{case}[c,d] holds.

\item
If $G(Y_i)\cap\xi_5\neq\emptyset$, 
then $G(Y_i)\in\{\xi_4,\xi_5,\xi_5\cup\{\l_0+\l_1-\p_1\}\}$ 
by \LEM{fix}b and one of \TAB{case}[g,h] holds. 

\end{Mlist}
See \EXM{l}, \EXM{h} and \EXM{j} for possible scenarios. 
We describe two additional scenarios to clarify how 
the above observations are applied.

For example, if $G(Y_\ell)=\{\e_0\}$ and $\mu_{\ell-1}$ 
contracts at least one real $(-1)$-curve, then $G(Y_{\ell-1})=\xi_1$ 
and thus $G(Y_{0})=\xi_1$ so that we are in case \TAB{case}b where $\alpha=\ell-1$.

For another example, suppose that $G(Y_\ell)=\{\l_0+\l_1\}$ and $\h_\ell=\l_0+\l_1$. 
If $\mu_{\ell-1}$ 
contracts exactly one real $(-1)$-curve, then $G(Y_{\ell-1})=\{\l_0+\l_1-\p_1\}$. 
If $\mu_{\ell-2}$ contracts at least one real $(-1)$-curve, 
then $G(Y_{\ell-2})=\xi_4$ so that we are in case \TAB{case}g.
If $\mu_{\ell-2}$ contracts only non-real complex $(-1)$-curves, 
then $h_\ell=5(\l_0+\l_1)-2\p_1-\p_2-\ldots-\p_r$ with $\sigma_*(\p_i)\neq \p_i$ for all $2\leq i\leq r$
so that $G(Y_{\ell-2})=\xi_5\cup\{\l_0+\l_1-\p_1\}$ 
and thus we are again in case \TAB{case}g.

It is straightforward to verify that each scenario is 
characterized by one of the cases in \TAB{case} and that 
for each case in \TAB{case}
there exists a scenario 
that realizes that case.
\end{proof}

\begin{proof}[Proof of \THM{coarse}.]
Suppose that
$\arrow{(Y_0,\h_0)}{\mu_0}{(Y_1,\h_1)}\arrow{}{\mu_1}{}\ldots\arrow{}{\mu_{\ell-1}}{(Y_\ell,\h_\ell)}$
is the unique adjoint chain of~$X$. 
Recall that $G(X)=G(Y_0)$ and that we can recover $\CG(X)$ from $G(X)$ (see \RMK{G} and \LEM{pp}c).
If $Y_\ell$ is not a weak del Pezzo surface, then 
we know from \LEM{bas} that $|G(X)|=1$ so that $\CG(X)$ is characterized by \THM{coarse}d.
We suppose in the remainder of the proof that $Y_\ell$ is a weak del Pezzo surface and that $\CG(X)\neq\emptyset$.
Hence, one of the cases in \LEM{G}[a-g] holds.
If $\set{g\in G(Y_\ell)}{ \k_\ell\cdot g=-2 }\neq\emptyset$, then 
$G(Y_\ell)=F_\R(Y_\ell)$ and thus, by \LEM{pp}b, $\CG(X)$ is characterized by \THM{coarse}a.
Now suppose that $\set{g\in G(Y_\ell)}{ \k_\ell\cdot g=-2 }=\emptyset$.
It follows from \LEM{case} that $G(Y_0)$ is either 
$\{\e_0\}$,                      
$\xi_1$,                         
$\xi_2$,                         
$\xi_2'\cup\{\e_0\}$,            
$\xi_2\cup\xi_1$,                
$\xi_3$                          
$\xi_3'\cup\{\e_0\}$,            
$\xi_3\cup\xi_1$,                
$\{\l_0+\l_1\}$,                 
$\{\l_0+\l_1-\p_1\}$,            
$\xi_4$,                         
$\xi_5\cup\{\l_0+\l_1-\p_1\}$ or 
$\xi_5$                          
so that 
$\CG(X)$ is characterized
by
(c),
(d),
(e),
(k),
(h),
(f),
(l),
(i),
(b),
(c),
(d),
(k) and
(g/j)
in \THM{coarse}, \resp.
\end{proof}

\begin{proof}[Proof of \COR{labels} and \COR{dis}.]
Direct consequence of \THM{nsl} (see also \TAB{F} and \RMK{G}).
For the tedious calculations and checking each row of \TAB{nsl}
for \COR{dis}, we used \citep[{\tt ns\_lattice}]{ns_lattice}.
\end{proof}

\begin{proof}[Proof of \COR{sphere}.]
Notice that \THM{coarse}b must hold and that $G(Y_0)=\{\l_0+\l_1\}$ so that $Y_0\in\BS_0$ by \LEM{case}.
If $Y_0\in \BS_0$, then $\h_0=t\,\l_0 + t\,\l_1$ for some $t\in\Z_{>0}$ and 
thus $Y_0$ is biregular isomorphic to a smooth non-ruled quadric.
We concluded the proof as all such quadrics are projectively equivalent.
Alternatively, see \citep[Corollary~1a]{nls-algo-min-fam}. 
\end{proof}

\begin{proof}[Proof of \THM{pmz}.]
Let $\varphi\c Y\to X$ be the smooth model of~$X\subset\P^n$ and let $\h,\k\in N(X)$
be the class of hyperplane sections and the canonical class, \resp.
Recall from \RMK{G} that two disjoint vertices correspond to 
simple families with classes $f,g \in G(X)$ \st $f\cdot g=1$.
It follows from \THM{coarse} that the vertices must have label 1 so that
$h^0(f)=h^0(g)=2$. 
We know from \LEM{pp}c that $-\k\cdot g=-2$ and 
since $p_a(f)=p_a(g)=0$ it follows from (AF) that $f^2=g^2=0$.
Thus the fibers of the associated maps $\varphi_f\c Y\to \P^1$
and $\varphi_g\c Y\to\P^1$ are $\varphi$-preimages of simple curves on $X$ in
their respective families.
Let $\nu\c Y\to\P^1\times\P^1$
be the morphism defined by $x\mapsto\bigl(\varphi_f(x),\varphi_g(x)\bigr)$.
The map $\nu$ is birational, since the general fiber of $\varphi_f$ and
$\varphi_g$ intersect in one point.  
Thus the inverse of $\nu$ composed with $\varphi$ defines a birational map  
$\P^1\times\P^1\dto X$ of bidegree $(d,d)$,
where $d$ is the degree of a simple curve on~$X$.
A birational map $\tau\c\P^1\times\P^1\dto X$ of 
bidegree $(a,b)$ defines two covering families of rational curves 
with members 
$\tau(\P^1\times\P^1\cap\{t\}\times\P^1)$
of degree $b$
and 
$\tau(\P^1\times\P^1\cap\P^1\times\{t\})$
of degree $a$, \resp.
This concludes the proof as $d$ is by definition 
the lowest possible degree for a rational curve
in a covering family. 
\end{proof}

\section{A characterization of hexagonal webs}

In this section we prove \THM{hex}. 

\begin{definition}
\label{def:hex}
Suppose that $X$ is a surface 
with real structure~$\sigma$.
Let~$\CW$ be a set of curves in~$X$
and let $\CW_p:=\set{C\in\CW}{p\in C}$ for all~$p\in X$.
Let $\CH(\CW)$ be defined as the graph with 
vertex set~$\set{p\in X}{|\CW_p|=3 \text{ and } \sigma(p)=p}$
and labeled edge set
\[
\set{(\{v,w\},C)}{C\in\CW \text{ and } v,w\in C \text{ are pairwise distinct}}.
\]
We call $\CW$ a \df{hexagonal web} if 
the vertex set of $\CH(\CW)$ is not contained in a reducible curve
and if a general edge $\{p,q\}$ of~$\CH(\CW)$
is contained in a subgraph as defined in \FIG{hex-graph},
where the edge-labels $A,B,C,D,E,F,G,H,I\in\CW$ are pairwise distinct.
\END
\end{definition}

\begin{figure}[!ht]
\centering
\begin{tikzpicture}[scale=1.5]
\draw[black,          thick] (-1,0)--(-0.5,1) node [midway, left] {\footnotesize$H$} 
                                   --(0.5,1) node [midway, above] {\footnotesize$F$}
                                   --(1,0) node [midway, right] {\footnotesize$D$}
                                   --(0.5,-1) node [midway, right] {\footnotesize$E$}
                                   --(-0.5,-1) node [midway, below] {\footnotesize$G$}
                                   --(-1,0) node [midway, left] {\footnotesize$I$};
\draw[black!20!green, very thick] (-1,0) --(0,0) node[midway, above=-2pt] {\footnotesize$A$} --(1,0) node[midway, below=-2pt] {\footnotesize$A$};
\draw[red,            very thick] (0.5,1)--(0,0) node[midway, right=-2pt] {\footnotesize$C$} --(-0.5,-1) node[midway, left=-2pt] {\footnotesize$C$};
\draw[cyan,           very thick] (-0.5,1)--(0,0) node[midway, right=-2pt] {\footnotesize$B$} --(0.5,-1) node[midway, left=-2pt] {\footnotesize$B$}; 
\draw[black,fill=white] ( 0  , 0) circle [radius=0.2] node {$p$};
\draw[black,fill=white] (-1  , 0) circle [radius=0.2] node {$ $};
\draw[black,fill=white] ( 1  , 0) circle [radius=0.2] node {$q$};
\draw[black,fill=white] (-0.5, 1) circle [radius=0.2] node {$ $};
\draw[black,fill=white] ( 0.5, 1) circle [radius=0.2] node {$ $};
\draw[black,fill=white] (-0.5,-1) circle [radius=0.2] node {$ $};
\draw[black,fill=white] ( 0.5,-1) circle [radius=0.2] node {$ $};
\end{tikzpicture}
\caption{See \DEF{hex}.}
\label{fig:hex-graph}
\end{figure}

\begin{remark}
\label{rmk:hex}
Suppose that $\CW$ is a hexagonal web
and that $\{p,q\}$ in \FIG{hex}a corresponds to a general edge of $\CH(\CW)$.
In \FIG{hex}b we draw all the curves in~$\CW_p\cup\CW_q$
and we obtain at least two new intersection points $r$ and $s$.
In \FIG{hex}c we draw all the curves in $\CW_r\cup\CW_s$
and we obtain again at least two new intersection points. 
We repeat the last step one more time so that 
we obtain a closed hexagon as in \FIG{hex}d. 
\FIG{hex}e is an example of a non-hexagonal web.  
We refer to \cite{nil1} for more information.
\END
\end{remark}

\begin{figure}[!ht]
\centering
\newcommand{\pp}[1] {\draw[draw=black, fill=gray!20, line width=0.1mm] #1 circle [radius=0.5mm];}

\def\tt{0.3}

\newcommand{\cL}[5] {\draw[cyan]           plot [smooth, tension=\tt] coordinates {#1 #2 #3 #4 #5};}
\newcommand{\dL}[4] {\draw[cyan]           plot [smooth, tension=\tt] coordinates {#1 #2 #3 #4};}
\newcommand{\eL}[3][] {\draw[cyan] #2 to [#1] #3;}
\newcommand{\cR}[5] {\draw[red]            plot [smooth, tension=\tt] coordinates {#1 #2 #3 #4 #5};}
\newcommand{\dR}[4] {\draw[red]            plot [smooth, tension=\tt] coordinates {#1 #2 #3 #4};}
\newcommand{\eR}[3][] {\draw[red] #2 to [#1] #3;}
\newcommand{\cH}[6] {\draw[black!20!green] plot [smooth, tension=\tt] coordinates {#1 #2 #3 #4 #5 #6};}
\newcommand{\dH}[4] {\draw[black!20!green] plot [smooth, tension=\tt] coordinates {#1 #2 #3 #4};}
\newcommand{\eH}[3][] {\draw[draw=black!20!green] #2 to [#1] #3;}

\newcommand{\cHH}[6] {\draw[black!20!green,densely dotted] plot [smooth, tension=\tt] coordinates {#1 #2 #3 #4 #5 #6};}

\def\d{0.05}

\def\Ax{-1              }\def\Ay{1              }
\def\Bx{-0.5            }\def\By{1              }
\def\Cx{0               }\def\Cy{1              }
\def\Dx{0.5-\d-\d       }\def\Dy{1              }
\def\Ex{1-\d-\d         }\def\Ey{1              }
\def\Fx{-1+\d+\d        }\def\Fy{-1             }
\def\Gx{-0.5+\d+\d      }\def\Gy{-1             }
\def\Hx{0               }\def\Hy{-1             }
\def\Ix{0.5-\d-\d-\d    }\def\Iy{-1             }
\def\Jx{1               }\def\Jy{-1             }
\def\Kx{-1              }\def\Ky{0.5+\d+\d      }
\def\Lx{-1              }\def\Ly{0.25+\d+\d     }
\def\Mx{-1              }\def\My{0+\d+\d        }
\def\Nx{-1              }\def\Ny{-0.25-\d-\d    }
\def\Ox{-1              }\def\Oy{-0.5           }
\def\Px{1               }\def\Py{0.5+\d+\d      }
\def\Qx{1               }\def\Qy{0.25-\d        }
\def\Rx{1               }\def\Ry{0              }
\def\Sx{1               }\def\Sy{-0.25          }
\def\Tx{1               }\def\Ty{-0.5           }
\def\Ux{-0.25-\d        }\def\Uy{0.25+\d+\d     }
\def\Vx{0.25-\d         }\def\Vy{0.25+\d        }
\def\Wx{0.5             }\def\Wy{0              }
\def\Xx{0.25            }\def\Xy{-0.25          }
\def\Yx{-0.25+\d        }\def\Yy{-0.25-\d       }
\def\Zx{-0.5-\d            }\def\Zy{0              }
\def\ZZx{\Zx}\def\ZZy{\Zy}
\def\OOx{0}\def\OOy{0}

\def\A{(\Ax,\Ay)}
\def\B{(\Bx,\By)}
\def\C{(\Cx,\Cy)}
\def\D{(\Dx,\Dy)}
\def\E{(\Ex,\Ey)}
\def\F{(\Fx,\Fy)}
\def\G{(\Gx,\Gy)}
\def\H{(\Hx,\Hy)}
\def\I{(\Ix,\Iy)}
\def\J{(\Jx,\Jy)}
\def\K{(\Kx,\Ky)}
\def\L{(\Lx,\Ly)}
\def\M{(\Mx,\My)}
\def\N{(\Nx,\Ny)}
\def\O{(\Ox,\Oy)}
\def\P{(\Px,\Py)}
\def\Q{(\Qx,\Qy)}
\def\R{(\Rx,\Ry)}
\def\S{(\Sx,\Sy)}
\def\T{(\Tx,\Ty)}
\def\U{(\Ux,\Uy)}
\def\V{(\Vx,\Vy)}
\def\W{(\Wx,\Wy)}
\def\X{(\Xx,\Xy)}
\def\Y{(\Yx,\Yy)}
\def\Z{(\Zx,\Zy)}
\def\OO{(\OOx,\OOy)}
\def\ZZ{(\ZZx,\ZZy)}

\def\CAP{(0,-1.2)}

\begin{tabular}{ccccc}
\begin{tikzpicture}
\cHH \R \W \OO \Z \ZZ \M;
\pp{\OO};\pp{\W}
\node at (\OOx,\OOy-0.3) {\footnotesize$p$};
\node at (\Wx,\Wy-0.3) {\footnotesize$q$};
\node[white] at \CAP {\tiny hexagonal web};
\end{tikzpicture}
&
\begin{tikzpicture}
\cL \A \U \OO \X \J;
\dL \B \V \W \T;
\cR \E \V \OO \Y \F;
\dR \P \W \X \G;
\cH \R \W \OO \Z \ZZ \M;
\draw[line width=0.7mm, draw=cyan] \W -- \V;
\draw[line width=0.7mm, draw=red]  \W -- \X;
\pp{\OO};\pp{\W};
\pp{\V};\pp{\X};
\node[white] at \CAP {\tiny hexagonal web};
\end{tikzpicture}
&
\begin{tikzpicture}
\cL \A \U \OO \X \J;
\dL \B \V \W \T;
\cR \E \V \OO \Y \F;
\dR \P \W \X \G;
\cH \R \W \OO \Z \ZZ \M;
\dH \Q \V \U \L;
\dH \S \X \Y \N;
\draw[line width=0.7mm, draw=cyan] \W -- \V;
\draw[line width=0.7mm, draw=red]  \W -- \X;
\draw[line width=0.7mm, draw=black!20!green] \V -- \U;
\draw[line width=0.7mm, draw=black!20!green] \X -- \Y;
\pp{\OO};\pp{\W};
\pp{\V};\pp{\X};
\pp{\U};\pp{\Y};
\end{tikzpicture}
&
\begin{tikzpicture}
\cL \A \U \OO \X \J;
\dL \B \V \W \T;
\dL \K \Z \Y \I;

\cR \E \V \OO \Y \F;
\dR \D \U \ZZ \O;
\dR \P \W \X \G;

\cH \R \W \OO \Z \ZZ \M;
\dH \Q \V \U \L;
\dH \S \X \Y \N;

\draw[line width=0.7mm, draw=cyan] \W -- \V;
\draw[line width=0.7mm, draw=red]  \W -- \X;
\draw[line width=0.7mm, draw=black!20!green] \V -- \U;
\draw[line width=0.7mm, draw=black!20!green] \X -- \Y;
\draw[line width=0.7mm, draw=red]  \U -- \ZZ;
\draw[line width=0.7mm, draw=cyan]  \Y -- \Z;

\pp{\OO};\pp{\W};
\pp{\V};\pp{\X};
\pp{\U};\pp{\Y};
\pp{\Z};\pp{\ZZ};

\end{tikzpicture}
&
\def\Ax{-1              }\def\Ay{1               }%
\def\Bx{-0.5            }\def\By{1               }%
\def\Cx{0               }\def\Cy{1               }%
\def\Dx{0.5-\d-\d       }\def\Dy{1               }%
\def\Ex{1-\d-\d-\d-\d   }\def\Ey{1               }%
\def\Fx{-1+\d+\d+\d+\d  }\def\Fy{-1              }%
\def\Gx{-0.5+\d+\d+\d+\d}\def\Gy{-1              }%
\def\Hx{0               }\def\Hy{-1              }%
\def\Ix{0.5-\d-\d-\d    }\def\Iy{-1              }%
\def\Jx{1               }\def\Jy{-1              }%
\def\Kx{-1              }\def\Ky{0.5+\d+\d       }%
\def\Lx{-1              }\def\Ly{0.25+\d+\d+\d+\d}%
\def\Mx{-1              }\def\My{0+\d+\d         }%
\def\Nx{-1              }\def\Ny{-0.25-\d-\d     }%
\def\Ox{-1              }\def\Oy{-0.5            }%
\def\Px{1               }\def\Py{0.5+\d+\d+\d+\d }%
\def\Qx{1               }\def\Qy{0.25-\d         }%
\def\Rx{1               }\def\Ry{0               }%
\def\Sx{1               }\def\Sy{-0.25           }%
\def\Tx{1               }\def\Ty{-0.5            }%
\def\Ux{-0.25-\d        }\def\Uy{0.25+\d+\d      }%
\def\Vx{0.25-\d         }\def\Vy{0.25+\d         }%
\def\Wx{0.5             }\def\Wy{0               }%
\def\Xx{0.25            }\def\Xy{-0.25           }%
\def\Yx{-0.25+\d        }\def\Yy{-0.25-\d        }%
\def\Zx{-0.5+\d+\d      }\def\Zy{0-\d            }%
\def\ZZx{-0.5-\d-\d-\d}\def\ZZy{0}%
\def\OOx{0}\def\OOy{0}%
\begin{tikzpicture}
\cL \A \U \OO \X \J;
\dL \B \V \W \T;
\dL \K \Z \Y \I;
\eL[out=-55, in=135]{\M}{(\Hx,\Hy)};
\eL[out=-45, in=135]{\C}{(\Rx,\Ry)};

\cR \E \V \OO \Y \F;
\dR \D \U \ZZ \O;
\dR \P \W \X \G;
\eR[out=200, in=45]{\R}{(\Hx,\Hy)};
\eR[out=220, in=45]{\C}{(\Mx,\My)};

\cH \R \W \OO \Z \ZZ \M;
\dH \Q \V \U \L;
\dH \S \X \Y \N;
\eH[out=190, in=15]{\P}{(\Kx,\Ky)};
\eH[out=190, in=-25]{\T}{(\Ox,\Oy)};

\draw[line width=0.7mm, draw=cyan] \W -- \V;
\draw[line width=0.7mm, draw=red]  \W -- \X;
\draw[line width=0.7mm, draw=black!20!green] \V -- \U;
\draw[line width=0.7mm, draw=black!20!green] \X -- \Y;
\draw[line width=0.7mm, draw=red]  \U -- \ZZ;
\draw[line width=0.7mm, draw=cyan]  \Y -- \Z;  

\pp{\OO};\pp{\W};
\pp{\V};\pp{\X};
\pp{\U};\pp{\Y};
\pp{\Z};\pp{\ZZ};
\end{tikzpicture}
\\
{\bf a} & {\bf b} & {\bf c} & {\bf d} & {\bf e}
\end{tabular}
\caption{See \RMK{hex}.}
\label{fig:hex}
\end{figure}

\newpage
\begin{definition}
\label{def:chex}
For the complex analogue of \DEF{hex}
we suppose that $X$ is a complex surface. 
Let~$\CW$ be a set of complex curves in~$X$
and as before let $\CW_p:=\set{C\in\CW}{p\in C}$ for all~$p\in X$.
Let $\CH(\CW)$ be defined as the graph with 
vertex set~$\set{p\in X}{|\CW_p|=3}$ and 
labeled edge set
$\set{(\{v,w\},C)}{C\in\CW \text{ and } v,w\in C \text{ are pairwise distinct}}$.
We call $\CW$ a \df{complex hexagonal web} if 
the vertex set of~$\CH(\CW)$ is not contained in a complex reducible curve
and if a general edge $\{p,q\}$ of~$\CH(\CW)$
is contained in a subgraph as defined in \FIG{hex-graph},
where the edge-labels $A,B,C,D,E,F,G,H,I\in\CW$ are pairwise distinct.
\END
\end{definition}

\begin{lemma}
\label{lem:pen}
Suppose that $Y\to X$ is the smooth model of a weak del Pezzo surface~$X$
\st $G(X)=F_\R(X)$.
If there exists $u,v,w\in G(X)$ 
\st $u\cdot v=v\cdot w=u\cdot w=1$, 
then
there exists a complex birational morphism $\chi\c Y\to\P^2$
such that images of simple curves with class $u$, $v$ or $w$
define three pencils of lines in $\P^2$.
\end{lemma}

\begin{proof}
Since $|G(X)|\geq 3$, we find that $1\leq \k^2\leq 6$ by \LEM{G}.

If $\k^2=6$, then $Y$ is the complex blowup of $\P^2$ in three points
and admits exactly three simple families.
The pullback of a pencil of lines through any center of blowup defines
a simple family and thus we concluded the proof for this case.

If $1\leq \k^2\leq 5$,
then by \LEM{orth}c there exists a complex $(-1)$-curve 
that is orthogonal to $u$, $v$ and $w$.
We contract it and by \LEM{blowdown} we obtain 
a complex weak del Pezzo surface $Y'$ of degree one less.
Moreover, by \PRP{class}c, the pushforward of the simple families 
with classes
$u$, $v$ and $w$ are complex simple families of $Y'$
\wrt an anticanonical embedding.
We repeat the same argument until we obtain a 
complex weak del Pezzo surface $Y''$ of canonical degree 6. 
We can now conclude the proof of this lemma, since we define
$\xi_\C$ as the composition of the birational morphism $Y\to Y''$
with the complex birational morphism $Y''\to\P^2$ which is the 
complex blowup of three complex points 
in the plane.
\end{proof}

\begin{lemma}
\label{lem:hex}
Suppose we are given simple families on a surface $X$.
If there exists a complex birational map $\rho\c X\dto\P^2$
such that the images of curves in the given simple families 
form three complex pencils of lines in $\P^2$, then
these simple families form a hexagonal web.
\end{lemma}

\begin{proof}
Suppose that  
$\CL:=\set{L\subset\P^2}{L\text{ is a complex line such that } L\cap\{u,v,w\}\neq\emptyset}$
for some complex points~$u,v,w\in\P^2$.
It follows from \THM{graf} that $\CL$ is a complex hexagonal web, 
since three complex pencils of lines in~$\P^2$ 
correspond to three complex lines in~$\P^{2*}$ (see \FIG{vero}).
The union of the members of the simple families form a set~$\CV$  
\st $\set{\rho(C)^-}{C\in\CV}\subseteq\CL$, where ${\cdot}^-$
denotes the Zariski closure.
Since $\rho$ is birational there exists 
a Zariski open subset $U\subset X$ \st $\rho|_U$ 
is a complex isomorphism.
We follow the procedure in \RMK{hex} and when we draw 
distinct simple curves $C,C'\in \CV$
we also draw complex lines $\rho(C)^-,\rho(C')^-\in\CL$.
We have $|\rho(C\cap C'\cap U)|\leq 1$
and thus if $|C\cap C'\cap U|\neq 0$,
then $C\cap C'\cap U=\{p\}$ and $\sigma(p)=p$.
Hence there exists in this case a simple curve $C''\in\CV$
\st $p\in C''$ and $C''\notin\{C,C'\}$. 
Therefore the procedure results in a closed hexagon so that
$\CV$ must be a hexagonal web.
\end{proof}

\begin{proof}[Proof of \THM{hex}.]
Suppose that
$\arrow{(Y_0,\h_0)}{\mu_0}{(Y_1,\h_1)}\arrow{}{\mu_1}{}\ldots\arrow{}{\mu_{\ell-1}}{(Y_\ell,\h_\ell)}$
is the adjoint chain of~$X$. 
Recall that $G(Y_0)$ uniquely determines $\CG(X)$ (see \RMK{G} and \LEM{pp}c).
It follows from \LEM{bas} that $Y_\ell$ must be a weak del Pezzo surface.
We call a curve $C\subset Y_i$ 
on the abstract surface $Y_i$ \df{simple} if its class is in $G(Y_i)$.
Since $Y_0\to X$ is the smooth model, there exists a birational map $\varphi\c X\dto Y_0$
which sends simple curves on $X$ to simple curves on $Y_0$.

We will make a case distinction.
For each of the cases we construct a complex birational map $X\dto \P^2$
that sends the curves in the three simple families of~$X$, 
that define three mutually disjoint vertices in~$\CG(X)$,
to three pencils of lines in $\P^2$. 
Such a map satisfies the hypothesis of \LEM{hex} 
so that the simple families form a hexagonal web and therefore 
its existence concludes the proof.

Suppose that 
$\set{f\in G(Y_\ell)}{ \k_\ell\cdot f=-2}\neq\emptyset$
so that we are in case \LEM{bas}a.
It follows from \LEM{G} that $G(Y_\ell)=F_\R(Y_\ell)$ and thus $\CG(X)$ is characterized by \THM{coarse}a.
Notice that there exists classes $u,v,w\in F_\R(Y_\ell)$ \st $u\cdot v=u\cdot w=v\cdot w=1$.
By \LEM{pen} there exists a complex birational morphism $\chi\c Y_\ell\to\P^2$ that 
sends a simple curve with class $u$, $v$ or $w$ to one of three pencils of lines in~$\P^2$.
As a straightforward consequence of the definitions, the birational morphism
$\mu_{\ell-1}\circ\ldots\circ \mu_0$ maps simple curves of $Y_0$ to simple curves of $Y_\ell$.
Hence, the composition $\chi\circ\mu_{\ell-1}\circ\ldots\circ \mu_0\circ\varphi\c X\dto \P^2$
satisfies the hypothesis of \LEM{hex}.

Now suppose that $\set{f\in G(Y_\ell)}{ \k_\ell\cdot f=-2}=\emptyset$.
Notice that there exists $u,v,w\in G(Y_0)$ \st $u\cdot v=u\cdot w=v\cdot w=1$.
It follows from \LEM{case} that one of the cases at \TAB{case}[a,c,d,g,h]
holds and $\{u,v,w\}$ is a subset of either $\xi_1$, $\xi_4$ or $\xi_5$.
We use the notation of \TAB{case}.

If \TAB{case}a holds, then the birational map $\mu_{\ell-1}\circ\ldots\circ \mu_0\circ\varphi: X\dto \P^2$
satisfies the hypothesis of \LEM{hex}. 

Suppose that \TAB{case}c holds.
Notice that in this case $\beta>\alpha$.
We may assume up to permutation of $(\e_j)_{j>0}$ that $(u,v,w)=(\e_0-\e_7,\e_0-\e_8,\e_0-\e_9)$
(we chose indices 7,8 and 9 in accordance with \EXM{h}). 
It follows from \LEM{bas}b that there exists a birational map 
$\eta\c Y_\beta\dto \P^2$ that sends the curves $\set{C\subset Y_\beta}{[C]\in \{\e_0-\e_7,\e_0-\e_8,\e_0-\e_9\}}$
to three pencils of lines in $\P^2$.
It follows that $\eta\circ\mu_{\beta-1}\circ\ldots\circ \mu_0\circ\varphi: X\dto \P^2$
satisfies the hypothesis of \LEM{hex}. 

Suppose that \TAB{case}g holds \st $Y_\alpha\in \BS_1$. 
We may assume up to permutation of $(\p_j)_{j>0}$ that 
$(u,v,w)=(\l_0+\l_1-\p_1-\p_2,\l_0+\l_1-\p_1-\p_3,\l_0+\l_1-\p_1-\p_4)$.
Notice that $\P^1\times\P^1$ blown up in one point is complex isomorphic to 
$\P^2$ blownup in two points. 
Thus there exists a complex isomorphism $\psi_1\c Y_\alpha\to Z$ \st $Z\in \BP_1$
and \st
$\psi_{1*}(\l_0+\l_1-\p_1)=\e_0$, 
$\psi_{1*}(\l_0-\p_1)=\e_1$,
$\psi_{1*}(\l_1-\p_1)=\e_2$ and
$\psi_{1*}(\p_j)=\e_{j+1}$ for $j\geq 2$.
Thus $\psi_{1*}(\{u,v,w\})=\{\e_0-\e_3,\e_0-\e_4,\e_0-\e_5\}$
where $\e_0$ is the class of a preimage of a line
along a birational morphism $\psi_2\c Z\to \P^2$.
It follows that the complex birational map
$\psi_2\circ\psi_1\circ\mu_{\alpha-1}\circ\ldots\circ \mu_0\circ\varphi: X\dto \P^2$
satisfies the hypothesis of \LEM{hex}. 

Cases \TAB{case}d and \TAB{case}h are analogous to \TAB{case}c and \TAB{case}g, \resp.
Since we considered all cases we concluded the proof of \THM{hex}.
\end{proof}

\newpage
\begin{proof}[Proof of \COR{hex}.]
By assumption, $X\subset\P^n$ contains at least two conics through each point,
and thus we know from \cite{sch6} (alternatively see \citep[Corollary~2]{nls-algo-min-fam}),
that $X$ is either a ruled surface of degree at most $3$
or a weak del Pezzo surface of canonical degree $3\leq \k^2 \leq 9$ with $\h\in\{-\frac{1}{2}\k, -\frac{1}{3}\k, -\frac{2}{3}\k, -\k \}$
the class of hyperplane sections.
The linear projection of a hexagonal web is again hexagonal
and thus we only have to classify linear normalizations $X_N\subset\P^m$.

If $X$ is a cubic ruled surface, 
then $X_N\subset\P^4$ is parametrized by the map
$(1:t:t^2:s:st)$ with $t,s\in\R$ \citep[Section~3]{sch6}. 
This birational map gives, for all $\beta\in\R$, 
locally an isomorphism between a pencil of conics on $X_N$ and
a pencil of lines in $\R^2$ defined by 
$s=\alpha\, t + \beta$ with parameter $\alpha \in\R$.
Therefore, by \THM{graf}, $X$ can be covered by an hexagonal web of conics.

Suppose that $X$ is a weak del Pezzo surface. 
If $8\leq \k^2 \leq 9$, 
then $X_N$ must be either a quadric surface, the plane or a Veronese surface,
since $X$ contains at least three conics through a general point. 
Each of these surfaces can be realized as a hexagonal web (see \FIG{vero}).
If $3\leq \k^2 \leq 6$, then $\h=-\k$, $X_N\subset\P^{\k^2}$ and by \THM{nsl}
we can construct a weak del Pezzo surface of degree $\k^2$
\st $|G(X)|\geq 3$ and $\CG(X)$ contains three vertices that do not share an edge. 
Thus \COR{hex}e is a consequence of \THM{hex}.
\end{proof}

\begin{table}[!ht] 
\setstretch{1.5}
\caption{Classification of Neron-Severi lattices of weak del Pezzo surfaces (see \THM{nsl}).}
\label{tab:nsl}
\begin{Mlist}

\item Each row in the table below after the dictionary
(except for the 5 rows with the $\times$-symbol)
encodes the Neron-Severi lattice $\BN(X)$ of a weak del Pezzo surface (see \DEF{nsl}).

\item The columns $d$, $D(A)$, $D(B)$, $\#E$, $\#G$ 
stand for $\k^2$, $\FD(A(X))$, $\FD(B(X))$, $|E_\R(X)|$ and $|G(X)|$, \resp. 
The entries of the columns $\FD(A(X))$ and $\FD(B(X))$ denote Dynkin types of 
the corresponding Dynkin diagrams.

\item
See \TAB{sub}, \PRP{class} and \PRP{sub} for the interpretation of the columns.
It follows from \LEM{G} that $G(X)=F_\R(X)$ except for the rows i, iii and 2.
By \RMK{G}, $G(X)$ uniquely determines the simple family graph~$\CG(X)$.

\item
There is not always a unique choice for the Dynkin type at column $D(A)$ (see \RMK{AA}).
We decorate a Dynkin type with $'$ if $\sigma_*(\e_0)\neq \e_0$.

\item 
For each component $W\subset B(X)$ (see \DEF{attr})
the Dynkin type of $\FD(W)$ at column $D(B)$
is \underline{solid underlined} 
if $\sigma_*(w)=w$ for all $w\in W$
and $\udot{\text{dashed underlined}}$ if $\sigma_*(W)=W$ but there exists $w\in W$ \st $\sigma_*(w)\neq w$.
If $W,W'\subset B(X)$ are different components \st $\sigma_*(W)=W'$, then
the corresponding Dynkin types are not decorated.

\item From the entry at column $\sigma_A||B$ we can recover $\sigma_*\c N(X)\to N(X)$
and $B(X)$ using the dictionary below.
For example, at row number 10, the entry $\frac{0}{0}\frac{0}{2}\frac{0}{1}\frac{0}{3}||\frac{3}{7}$
encodes that $\sigma_*(\e_0)=\e_0$, $\sigma_*(\e_1)=\e_2$, $\sigma_*(\e_2)=\e_1$, $\sigma_*(\e_3)=\e_3$
and $B(X)=\{\e_0-\e_1-\e_2-\e_3\}$ \wrt a type 1 basis $\md{\e_0,\ldots,\e_r}_\Z$ for $N(X)$ (see \DEF{type}).  
It follows from \RMK{A} that $A(X)=\{\e_1-\e_2\}$.
Indeed we verify that $\FD(A(X))=A_1$ and $B(X)=\underline{A_1}$.

\item We can uniquely recover $h^0\c N(X)\to\Z_{\geq0}$ from $B(X)$ (see \LEM{h0}) and thus 
the entries of the columns in a row accounts for all data of $\BN(X)$.

\item See \EXM{howto} for a scenario which explains how the table can be used.
\end{Mlist}
\end{table}

\begin{table}
\setstretch{1.8}
A dictionary for symbols in the column $\sigma_A||B$:
\tiny
\\[2mm]

\end{tabular}
\end{table}

\clearpage
\section{Acknowledgements}
I would like to thank H. Pottmann and M. Skopenkov
for interesting comments and discussions. 
The image of \FIG{cleb} is adapted from \citep[Wikipedia]{wiki}.
The images of surfaces were made using \citep[Povray]{povray} and 
the algorithms were implemented using \citep[Sage]{sage}.
Financial support was provided by the Austrian Science Fund (FWF): P33003.

\section{References}
\bibliography{geometry}

\paragraph{address of author:}~
Johann Radon Institute for Computational and Applied Mathematics (RICAM), Austrian Academy of Sciences
\\
\textbf{email:} niels.lubbes@gmail.com

\end{document}